\documentclass{aptpub}

\usepackage{amssymb,amsfonts,amsmath}
\usepackage{graphicx}
\usepackage{caption}
\usepackage{subcaption}
\usepackage{paralist}
\usepackage[pagebackref=true]{hyperref}
\usepackage{color}
\usepackage[normalem]{ulem}
\usepackage{stmaryrd} 
\usepackage{bbm}
\usepackage{mathtools}

\authornames{F.~BAERISWYL, V.~CHAVEZ-DEMOULIN, O.~WINTENBERGER} 
\shorttitle{Asymptotics of some Poisson cluster processes} 

\providecommand{\A}{}
\providecommand{\R}{}
\providecommand{\S}{}

\providecommand{\N}{}
\providecommand{\C}{}

\renewcommand{\A}{\mathbb{A}}
\renewcommand{\R}{\mathbb{R}}
\renewcommand{\S}{\mathbb{S}}

\renewcommand{\N}{\mathbb{N}}
\renewcommand{\C}{\mathbb{C}}

\renewcommand{\leq}{\leqslant}
\renewcommand{\geq}{\geqslant}

\newcommand*\diff{\mathop{}\!\mathrm{d}}

\newcommand{\E}[1]{{\mathbb E}\left[#1\right]}

\newcommand{\p}[1]{{\mathbb P}\left(#1\right)}

\newcommand{\I}[1]{{\mathbbm 1}_{\{#1\}}}

\newcommand{\Cprob}[2]{\mathbb{P}\left(\left. #1 \; \right| \; #2\right)}

\newcommand{\defeq}{\vcentcolon=}
\newcommand{\eqdef}{=\vcentcolon}

\newcommand\cA{\mathcal A}

\newcommand\cE{\mathcal E}

\newcommand\cO{\mathcal O}

\newcommand{\bX}{\mathbf{X}}
\newcommand{\bx}{\mathbf{x}}
\newcommand{\bt}{\mathbf{t}}
\newcommand{\bZ}{\mathbf{0}}
\newcommand{\bTheta}{\mathbf{\Theta}}

\begin{document}

\title{Tail asymptotics and precise large deviations \\ for some Poisson cluster processes} 

\authorone[Universit\'{e} de Lausanne and Sorbonne Universit\'{e}]{Fabien Baeriswyl} 
\authortwo[Universit\'{e} de Lausanne]{Val\'{e}rie Chavez-Demoulin} 
\authorthree[Sorbonne Universit\'{e}]{Olivier Wintenberger} 


\addressone{D\'{e}partement des Op\'{e}rations, Anthropole, CH-1015 Lausanne, Suisse} 
\emailone{fabien.baeriswyl@unil.ch, fabien.baeriswyl@sorbonne-universite.fr} 
\addresstwo{D\'{e}partement des Op\'{e}rations, Anthropole, CH-1015 Lausanne, Suisse}
\addressthree{Laboratoire de Probabilit\'{e}s, Statistique et Mod\'{e}lisation, Sorbonne Universit\'{e}, Campus Pierre et Marie Curie, 4 place Jussieu, 75005 Paris, France}

\begin{abstract}
    We study the tail asymptotics of two functionals (the maximum and the sum of the marks) of a generic cluster in two sub-models of the marked Poisson cluster process, namely the renewal Poisson cluster process and the Hawkes process. Under the hypothesis that the governing components of the processes are regularly varying, we extend results due to \cite{fgams06} and \cite{bwz19} notably, relying on Karamata's Tauberian Theorem to do so. We use these asymptotics to derive precise large deviation results in the fashion of \cite{km99} for the above-mentioned processes. 
\end{abstract}



\section{Introduction}
\label{sct:introduction}
In this paper, we study the asymptotic properties of processes exhibiting clustering behaviour. Such processes are common in applications: for instance, earthquakes in seismology,  where a main shock has the ability to trigger a series of secondary shocks in a specific spatio-temporal neighbourhood; but also accidents giving rise to a series of subsequent claims in non-life insurance or heavy rainfall in meteorology to name a few. We will focus on two different processes that have effectively been used in these fields. The Hawkes process has been introduced in the pioneer works of \cite{vo83} and \cite{ogata88}, and has found applications in earthquake modeling (see e.g. \cite{mvj92}), in finance (see e.g. \cite{cd05}, \cite{hawkes18}), in genome analysis (see \cite{rb10}) or in insurance (see \cite{swishchuk18}). The renewal Poisson cluster process is a tool of choice in an insurance context for modelling series of claims arising from a single event (see e.g. \cite{mikosch09} for a reference textbook),  as well as in teletraffic modelling (see \cite{fgams06}) and in meteorology and weather forecast (see e.g. \cite{ffgl87} or \cite{ock00}).

The above processes, described heuristically and in specific contexts above, are part of the class of the so-called point processes: for a comprehensive overview, see the monographs of \cite{dvj03} and \cite{dvj08} or, more recently, and with connection to martingale theory, see \cite{bremaud20}. Point process theory is an elegant framework describing the properties of random points occurring in general spaces. In both cases the temporal marked point process  $N$ possesses a representation as an infinite sum of Dirac measures (recall that the Dirac measure $\varepsilon$ on $\cA$ satisfies for every $A\in \cA$ that $\varepsilon_{x}(A) = 1$ if $x \in A$ and $\varepsilon_{x}(A) = 0$ otherwise): $$N(\cdot) = \sum_{i=1}^{\infty} \varepsilon_{T_i, A_i}(\cdot)$$ where $T_i$ is the (random) time of occurrence of the $i$th event and $A_i$ is its associated mark. The specific temporal marked point processes that we are interested in are cluster point processes. More specifically, we will assume that there exists an immigration process, under which independent points arise at a Poissonian rate; then, each of these immigrant events has the ability to trigger new points, called first generation offspring events. We will then look at two submodels. One is the renewal Poisson cluster process. It is complete with the immigrant events and their first generation offsprings. The term ``renewal" comes from the fact that the times of the events form a renewal sequence. The other submodel is the Hawkes process in which every point of the first generation has the ability to generate new points, acting as an immigrant event, potentially generating therefore a whole cascade of points. Each immigrant event and its associated offspring events (whether direct children or indirect) form a generic cluster. 

We will study the tail asymptotics of the partial maxima and sums of a transformation $X=f(A)$, for some nonnegative real valued function $f$,  of the mark $A$ of any event of $N$. Determining the behaviour of the maximum  and the sum  at the level of the cluster decomposition of a process is crucial to obtain limit theorems for partial maxima and sums of the whole process over finite intervals, see e.g. \cite{st10}, \cite{kz15} or \cite{bwz19}.  Thus, we describe first a generic cluster from each of the above-mentioned processes.  

For the renewal Poisson cluster process, we will consider a distributional representation of the maximum of the marks in the generic cluster, denoted $H^{R}$, $$H^{R} \eqdist X \vee \bigvee_{j=1}^{K_A} X_j$$ where $X$ is a transformation $f(A)$ of the mark $A$ of the immigrant event, and $X_j$ is the transformed mark of the $j$th first-generation offspring event. The number of offspring events, $K_A$, is random and possibly dependent on $X$. In particular, we will let the vector $(X, K_A)$ be heavy-tailed, and assess whether the heavy-tailedness transfers to $H^{R}$. Details are relegated to Section 2. Note that under the hypothesis that $X$ and $K_A$ are independent the above distributional equation has received early consideration, e.g. in \cite{tw18} or \cite{jm06}, where it is shown that $H^{R}$ and $X$ belong to the same maximum domain of attraction of some extreme value distribution (MDA for short - see \cite{resnick}, \cite{dhf06} or \cite{ekm13} for references on extreme value theory). A more recent advance in the case where $X$ and $K_A$ are dependent is to be found in \cite{bmz22}, where a similar conclusion is reached about the MDA. Our emphasis is on the Fr\'{e}chet MDA, which allows a certain refinement on the characterisation of the tail asymptotics. 

We will also consider tail asymptotics for the sum functional, which for the very same renewal Poisson cluster process, and for a generic cluster, possesses the distributional representation $$D^{R} \eqdist X + \sum_{j=1}^{K_A} X_j$$ supposing again that $(X, K_A)$ is heavy-tailed, We will also assess whether the heavy-tailedness of $(X, K_A)$ transfers to $D^{R}$. This equation has received consideration under the hypothesis that $X$ and $K_A$ are independent, see \cite{fgams06}. We will retrieve their results in our framework. More recently, in the case of arbitrary dependence between $X$ and $K_A$, similar asymptotics have been derived in \cite{oc21}.

We will then derive the very same kind of tail asymptotics, for the very same functionals of a generic cluster in the context of the Hawkes process. The distributional representation associated with the maximum of the marks in a generic cluster, denoted $H^{H}$, is given by $$ H^{H} \eqdist X \vee \bigvee_{j=1}^{L_A} H_j^{H}$$ where $L_A$ is the number of first-generation offspring events $A$ of the event acting as immigrant, and $H_j$ is the maximum of the marks of the offsprings of the $j$th offspring of the immigrant event considered, itself acting as immigrant for further subranches of the cluster, emphasising once again the cascade structure of the Hawkes process. The equation for $H^H$ above is a special case of the higher-order Lindley equation, see \cite{kks94}. Note that $X=f(A)$ and $L_A$ are dependent through $A$. Letting $L_A$ be Poisson distributed with parameter $\kappa_A$, and $(X, \kappa_A)$ be heavy-tailed we assess whether this transfers to $H^{H}$. This functional has received attention in the recent work of \cite{bmz22}, where it was shown that $H^{H}$ has the same MDA as that of $X$. 

The distributional representation associated with the sum of the marks in a generic cluster in the Hawkes process, denoted $D^{H}$, is given by $$D^{H} \eqdist X + \sum_{j=1}^{L_A} D_j^{H}.$$ We will again let $(X, \kappa_A)$ be heavy-tailed, and assess whether this transfers to $D^{H}$. This distributional equation, with cascade structure, has been extensively studied: see e.g. \cite{bkp13}; but also, as a main stochastic modelling approach to Google's PageRank algorithm, see \cite{lsv07}, \cite{joc10}, \cite{vl10}, \cite{cloc14}, \cite{cloc17} and, even more closely related to our results, \cite{oc21}; in the context of random networks, see \cite{mr20} or \cite{m23}; for a recent, theoretical advance as well as application to queuing systems, see \cite{af18} or \cite{eah18}.

The way we will deal with heavy-tailedness is through the classical notion of regular variation, introduced by J. Karamata in the 20th century (see e.g. \cite{k33}), which specifies that the functions of interest behave, in a neighbourhood of infinity, like power-law functions. For a thorough, textbook treatment of the topic in univariate settings, see \cite{bgt89}; we rely on \cite{resnick}, \cite{resnick2}, \cite{dhf06} and \cite{mw23} for the multivariate case. 

The flexibility offered by our approach to the way we specify the regular variation of the governing components of our processes allows us, in the sequel, to extend results due to \cite{rs08}, \cite{fgams06}, \cite{hs08} or \cite{dfk10}, that all studied the asymptotics of the tail of distributional quantities such as $H$ and $D$ in the above examples, but under various assumptions on the relations of the tails of $X$ and $K_A$ for the renewal Poisson cluster process, respectively $X$ and $L_A$ for the Hawkes process.  

Finally, we use the results on the tails of $H$ and $D$ to derive (precise) large deviation principles for our processes of interest, in the flavour of \cite{sn79}, \cite{mn98}. The ``precise" terminology comes from the fact that we have exact asymptotic equivalence instead of logarithmic ones when assuming Cram\'{e}r's condition. Early results on precise large deviations in the case of non-random maxima and sums can be found in \cite{n69}, \cite{n69-2}, \cite{heyde67} and \cite{ch91}. The case of random maxima and sums of extended regularly varying random variables (a class containing regularly varying random variables) is to be found in \cite{km99}, and we will rely on their results to derive our very own precise large deviation results. Contributions in this area for another subclass of subexponential distributions, namely the class of consistently varying random variables, can be found in \cite{tsjz01} or \cite{nty04}; for precise large deviations results on (negatively) dependent sequences, see \cite{tang06} or \cite{liu09}.  

The organisation of the paper is as follows: in Section~2, we describe the main processes of interest, that are part of the Poisson cluster process family; in Section~3, we recall some important notions and characterisations of (multivariate) regular variation; in Section 4, we derive the tail asymptotics for the maximum of the marks in a generic cluster in the renewal Poisson cluster process; in Section 5, we derive the tail asymptotics for the sum of the marks in a generic cluster in the renewal Poisson cluster process; in Section~6, we derive the tail asymptotics for the maximum of the marks in a generic cluster in the Hawkes process; in Section~7, we derive the tail asymptotics for the sum of the marks in a generic cluster in the Hawkes process; in Section~8, we use the results from Section~4 to Section~7 to derive (precise) large deviations results for our processes of interest. 

\section*{Notation}

Vectors are usually in boldface. By ``i.i.d." we classically mean independent and identically distributed and, consistently, ``i.d." means identically distributed. We let $\lceil \cdot \rceil$ denote the upper integer part, $\lfloor \cdot \rfloor$ the lower integer part. For two functions $f(\cdot)$ and $g(\cdot)$, and $c \in \{0, \infty\}$, we note $f(x) = \cO(g(x)), \text{ as } x \rightarrow c$ whenever $\limsup_{x \rightarrow c} \lvert f(x) / g(x) \rvert  \leq M$, for some finite $M > 0$; $f(x) = o(g(x)), \text{ as } x \rightarrow c$ whenever $\lim_{x \rightarrow c} \lvert f(x) / g(x) \rvert =0$;  $f(x) \sim g(x)$, as $x \rightarrow c$ whenever $\lim_{x \rightarrow c}  f(x) /g(x) = 1.$ The product of two measures $\mu$ and $\nu$ is written as the tensor product $\mu \otimes \nu$. 

\section{Random functionals of clusters}

We formally introduce the general Poisson cluster process, a class which includes the processes discussed in Section~\ref{sct:introduction}, keeping the spirit of the presentation and (most) notations from \cite{bwz19}. As hinted in Section~\ref{sct:introduction}, this process is made up of two components: an immigration process and an offspring process. 

The immigration process, say $N_0$, is a marked homogeneous Poisson process (or marked PRM in short, for marked Poisson random measure) with representation given by: $$N_0(\cdot) \defeq \sum_{i=1}^{\infty} \varepsilon_{\Gamma_i, A_{i0}}(\cdot).$$
This point process has mean measure $\nu \text{Leb} \otimes F$, for $\nu > 0$, on the space $[0, \infty) \times \A$, where $\text{Leb}$ is the Lebesgue measure, $F$ is the common distribution function to all marks $(A_{i0})_{i \in \N}$, which take values on a measurable space $(\A, \cA)$, and where $\cA$ corresponds to the Borel $\sigma$-field on $\A$. In particular, this means that the sequence of times $(\Gamma_i)_{i \in \N}$, corresponding to the arrivals of immigrant events, is a homogeneous Poisson process with rate given by $\nu \text{Leb}$. Since the space $\A$ can be quite general, applying a transformation $f(\cdot): \A \rightarrow \R_+$ is natural, especially in practical applications. Note that we will also assume this transformation of the marks, i.e. we only consider nonnegative transformed marks in our models. For example, in a non-life insurance context, supposing that $A_{i0}$ represents the characteristics of the $i$th accident, $f(A_{i0})$ could represent the claim size pertaining to this accident. In subsequent sections, and to ease the notation, we shall denote $X_{i0} \defeq f(A_{i0})$. 

Conditioning on observing an immigration event at time $\Gamma_i$, the marked PRM $N_0$ is supplemented with an additional point process in $M_p([0, \infty) \times \A)$ (the space of locally finite point measures on $[0, \infty) \times \A$) that we denote by $G_{A_{i0}}$. The cluster of points $G_{A_{i0}}$, occurring after time $\Gamma_i$, augments $N_0$ with triggered, offspring points or events.

The offspring cluster process, conditioned on observing an immigrant event $(\Gamma_{i}, A_{i0})$, admits the representation $$G_{A_{i0}}(\cdot) \defeq \sum_{j=1}^{K_{A_{i0}}} \varepsilon_{T_{ij}, A_{ij}}(\cdot)$$ where $(T_{ij})_{1 \leq j \leq K_{A_{i0}}}$ forms a sequence of nonnegative random variables indicating, for a fixed $j$, the random time from the immigrant event occurring at time $\Gamma_i$ and the $j$th event of the cluster, and where $K_{A_{i0}}$ is a random variable with values in $\N_0$, corresponding to the number of events in the $i$th cluster. These events are the offspring of the immigrant event identified by $(\Gamma_i, A_{i0})$. A complete representation of the general Poisson cluster process is given by
\begin{equation*} N(\cdot) \defeq \sum_{i=1}^{\infty} \sum_{j=0}^{K_{A_{i0}}} \varepsilon_{\Gamma_i + T_{ij}, A_{ij}}(\cdot) \end{equation*} providing we set $T_{i0}=0$ for all $i \in \N$. 

The first functional of interest is the maximum of the marks in the $i$th cluster, defined by \begin{equation}\label{eq:hi} H_i \defeq  \bigvee_{j=0}^{K_{A_{i0}}} X_{ij}. \end{equation}

Above, for ease of notation, we have defined $X_{i0}=f(A_{i0})$; accordingly, we let $X_{ij} = f(A_{ij})$ for the transformation $f(\cdot) : \A \rightarrow \R_{+}$. The point process associated with the $i$th cluster is defined by $$C_i(\cdot) \defeq \varepsilon_{0, A_{i0}}(\cdot) + G_{A_{i0}}(\cdot).$$ It allows us to define the second functional of interest in this paper, namely the sum of all marks in the $i$th cluster, by \begin{equation}\label{eq:di} D_i \defeq \int_{[0, \infty)\times \A} f(a) C_i(\diff t, \diff a).\end{equation}

In Section 8, we will look at the whole process on a subset of the temporal axis: at the level of the point process $N$, the sum of all marks in the finite time interval $[0, T]$, for $T > 0$, is given by \begin{equation}\label{eq:st} S_T \defeq \int_{[0, T]\times \A} f(a) N(\diff t, \diff a).\end{equation}

From Section~4 to Section~7, we propose tail asymptotics for $H_i$ and $D_i$ in the settings of mainly two different submodels of the general Poisson cluster process, briefly described in the introduction, that we formally discuss next, keeping the presentation in \cite{bwz19}, but fully described in Example 6.3 of \cite{dvj03}. However, we refer to the former reference for a complete description. In our work, we also assume that the sequence of marks $(X_{ij})$ is i.i.d.  

\subsection{Mixed binomial Poisson cluster process}
    In this model, the assumptions on $N_0$ are kept unchanged and the $i$th cluster has a representation of the form $$C_i(\cdot) = \varepsilon_{0, A_{i0}}(\cdot) + G_{A_{i0}}(\cdot) = \varepsilon_{0, A_{i0}}(\cdot) + \sum_{j=1}^{K_{A_{i0}}} \varepsilon_{W_{ij}, A_{ij}}(\cdot)$$ where $\left(K_{A_{i0}}, (W_{ij})_{j \geq 1}, (A_{ij})_{j \geq 0}\right)_{i \geq 0}$ is an i.i.d. sequence, the sequence $(A_{ij})_{j \geq 0}$ is also i.i.d. for any fixed $i=1, 2, \ldots$ and, finally, $(A_{ij})_{j \geq 1}$ is independent of both $K_{A_{i0}}$ and $(W_{ij})_{j \geq 1}$ for any $i = 1, 2, \ldots$. Note that this latter statement does not exclude dependence between $A_{i0}$ and $K_{A_{i0}}$ (respectively $(W_{ij})_{j \geq 1}$). Additionally, it is assumed that $\E{K_A} < \infty$, where $K_A$ denotes a generic random quantity distributed as $K_{A_{i0}}.$
\subsection{Renewal Poisson cluster process}

In this model, the $i$th cluster has the representation \begin{equation}\label{eq:rpcr} C_i(\cdot) = \varepsilon_{0, A_{i0}}(\cdot) + G_{A_{i0}}(\cdot) = \varepsilon_{0, A_{i0}}(\cdot) +  \sum_{j=1}^{K_{A_{i0}}} \varepsilon_{T_{ij}, A_{ij}}(\cdot)\end{equation} where all the assumptions from Section 2.1 hold, except that now, we denote the occurrence time sequence of the offspring events by $(T_{ij})_{j \geq 1}$ to emphasise that this forms a renewal sequence, that is, for any fixed $i = 1, 2, \ldots$, $T_{ij} = W_{i1} + \cdots + W_{ij}.$ Note that this process is such that every Poisson immigrant has only $K_{A_{i0}}$ first generation offspring events. These points cannot generate further generations themselves, in contrast with the Hawkes process that we will introduce next. 

Applying the transformation $f$ on the marks of the events, we will, in Section~4 and Section~5, derive tail the asymptotics of generic versions of Equation~\eqref{eq:hi} and Equation~\eqref{eq:di}, given by: 
\begin{enumerate}
    \item for the maximum, \begin{equation}\label{eq:rp-h} H^{R} \eqdist X \vee \bigvee_{j=1}^{K_{A}} X_j;\end{equation}
    \item for the sum, \begin{equation}\label{eq:rp-d} D^{R} \eqdist X + \sum_{j=1}^{K_{A}} X_j. \end{equation}
\end{enumerate} We isolate $X \defeq f(A)$ from the rest of the transformed claims $(X_j) \defeq \big(f(A_j)\big)$, to emphasise the possible dependence between $X$ and $K_{A}$. 

\begin{rem}\label{rem:simple-pro}
    These two processes have been considered in the monograph \cite{mikosch09}. The mixed binomial Poisson cluster process and the renewal Poisson cluster process are very similar in their description, and because their sole difference is the placement of the points along the time axis, we focus - in what follows - on the renewal Poisson cluster process. The results of Section 4 and Section 5 are directly applicable to the mixed binomial Poisson cluster process; the results of Section 8 also apply, upon the use of an alternative justification regarding the left-over effects to be discussed in that section. We refer to \cite{bwz19} and \cite{bmz22} for justifications. 
\end{rem}

\subsection{Hawkes process}

The specificity of the Hawkes process is that the clusters have a recursive pattern, in the sense that each point, whether immigrant or offspring, has the ability to act as an immigrant and generate a new cluster. To obtain the representation of the $i$th cluster $G_{A_i}$, one typically introduces a time shift operator $\theta_t$, as in \cite{bwz19}. Let $m(\cdot) = \sum_{j=1}^{\infty} \varepsilon_{t_j, a_j}(\cdot)$ be a point measure: then, the time-shift operator is defined by $$\theta_t m(\cdot) =  \sum_{j=1}^{\infty} \varepsilon_{t_j + t, a_j}(\cdot)$$ for all $t \geq 0$. Then, the (recursive) representation of the $i$th cluster, conditioning on observing an immigration event $(\Gamma_i, A_{i0}),$ is given by $$C_i(\cdot) = \varepsilon_{0, A_{i0}}(\cdot) + G_{A_{i0}}(\cdot) = \varepsilon_{0, A_{i0}}(\cdot) +  \sum_{j=1}^{L_{A_{i0}}} \left(\varepsilon_{\tau_{ij}^1, A_{ij}^1}(\cdot) + \theta_{\tau_{ij}^1} G_{A_{ij}^1}(\cdot) \right)$$ where, given $A_{i0}$, the first-generation offspring process $N_{A_{i0}}(\cdot) \defeq \sum_{j=1}^{L_{A_{i0}}} \varepsilon_{\tau_{ij}^1, A_{ij}^1}(\cdot) $ is again a Poisson process, this time with (random) mean measure $\int h(s, A_{i0}) \diff s \otimes F$, and where the sequence $(G_{A_{ij}^1})_{j \geq 1}$ is i.i.d. and independent of the first-generation offspring process $N_{A_{i0}}$. Note that the sequence of times in the cluster representation $G_{A_{i0}}$, hereby denoted as $(\tau_{ij})$, is the sequence of times of the first-generation offspring events. The function $h(\cdot)$ is referred to as the fertility function and controls both the displacement and the expected number of offspring(s) of a specific event. Hence, by definition, the number of first generation offspring events is Poisson and depends on the mark of the event acting as an immigrant to the stream of points considered. Note that the above representation also emphasises the independence between the subclusters considered at any point, from the immigrant perspective. There is a connection with Galton-Watson theory that was historically used to show that the Hawkes process is a general Poisson cluster process (see \cite{ho74}); we define it as part of this family, but the Hawkes process is classically introduced from the self-excitation perspective, that is, from the specification of the function $h(\cdot)$ (see e.g. \cite{hawkes71}). 

We propose in Section~6 and Section~7 tail asymptotics for the generic versions of Equation~\eqref{eq:di} and Equation~\eqref{eq:hi}, which satisfy, in the settings of the Hawkes process, fixed-point distributional equations of the form:
\begin{enumerate}
    \item for the maximum, \begin{equation}\label{eq:h-h} H^{H} \eqdist X \vee \bigvee_{j=1}^{L_A} H_j^{H} ;\end{equation}
    \item for the sum, \begin{equation}\label{eq:h-d} D^{H} \eqdist X + \sum_{j=1}^{L_A} D_j^{H}; \end{equation}
\end{enumerate}
where $L_A|A \sim \text{Poisson}(\kappa_A)$ and $\kappa_A = \int_{(0, \infty)} h(t,A) \diff t$ and where $(H_j^{H})$ and $(D_j^{H})$ are i.i.d. copies of $H^{H}$ and $D^{H}$, respectively. In this work, we always assume the subcriticality condition (in the terminology of branching processes) $\E{\kappa_A} < 1$, in order for clusters to be almost surely finite. This also implies that the expected total number of points in a cluster is given by $\frac{1}{1-\E{\kappa_A}}$, using a geometric series argument (see Chapter 12 in \cite{bremaud20}). As pointed out in \cite{af18} and references therein, the combination of the subcriticality assumption, the fact that the random quantities involved in Equation~\eqref{eq:h-d} are nonnegative and the assumption that $\E{X} < \infty$ (to be made through the index of regular variation of $X$ in further sections) yields the existence and uniqueness of a nonnegative solution to this distributional equation; for Equation~\eqref{eq:h-h}, a discussion about the existence of potentially multiple solutions to the higher-order Lindley equation can be found in \cite{bcop22}. Lastly, note that Equation~\eqref{eq:h-h} and Equation~\eqref{eq:h-d} emphasise the cascade structure of the Hawkes process. 

\section{A word on regular variation}

Throughout this paper, we will assume that the governing random components of our processes of interest are regularly varying, that is, roughly speaking, exhibit heavy tails. More specifically, we will assume that the random vector $\mathbf{X}$ is regularly varying. For the renewal Poisson cluster process, this amounts to assume that $\mathbf{X} = (X, K_A)$ is regularly varying, where $X$ and $K_A$ are defined as in Section 2.2; for the Hawkes process, this amounts to assume that $\mathbf{X} = (X, \kappa_A)$ is regularly varying, where $X$ and $\kappa_A$ are defined in Section 2.3. The exact definition of regular variation  varies in the literature depending on the context (see e.g. \cite{resnick}, \cite{resnick2}, \cite{dhf06}, \cite{hl06}, \cite{szm16}). Hence, we first recall the definition of regular variation we use in this text in full generality, borrowing notations from \cite{mw23}. We let $\R^d_{\bZ}=\R^d \backslash \{\bZ\}$ with $\bZ = (0,0,\ldots, 0)$. We let $\lvert \cdot \rvert$ be any norm on $\R^d$ (by their equivalence). Note that, in subsequent sections, our framework is restricted to the case where $d=2$. 

\defn{ Let $\bX$ be a random vector with values in $\R^d$. Suppose that $\lvert \bX \rvert$ is regularly varying with index $\alpha >0$. Let $(a_n)$ be a real sequence satisfying $n \p{\lvert \bX \rvert > a_n} \rightarrow 1$, as $n \rightarrow \infty$. The random vector $\bX$ (and its distribution) are said to be regularly varying if there exists a non-null Radon measure $\mu$ on the Borel $\sigma$-field of $\R^d_{\bZ}$ such that, for every $\mu$-continuity set $A$, it holds that $$\mu_n(A) \defeq n \p{a_n^{-1} \bX \in A} \rightarrow \mu(A), \text{ as } n \rightarrow \infty .$$}\label{def:rv}
In the above definition, two remarks are in order: 

\begin{enumerate}
	\item the regular variation of $\lvert \bX \rvert$ is univariate; standard definition applies, namely that the distribution of $\lvert \bX \rvert$ has power-law tails, that is, $\p{\lvert \bX \rvert > x} = x^{-\alpha} L(x)$ for $x>0$, where $L(\cdot)$ is a slowly varying function;
	\item the kind of convergence that takes place is vague convergence. The limiting measure possesses various nice properties, among which one can cite homogeneity: for any Borel set $B \subset \R^d_{\bZ}$ and $t>0$, it holds that $\mu(tB)=t^{-\alpha}\mu(B)$. 
\end{enumerate}

Rather than using the sequential form as in Definition~\eqref{def:rv}, it is possible to use an alternative continuous form. Additionally, a distinguished characterisation in the literature is through a limiting decomposition into ``spectral" and `radial" parts, see \cite{resnick2}.  

\begin{prop}[Theorem 6.1 in \cite{resnick2}]\label{prop:rvdecomp}
	A random vector $\bX$ with values in $\R^d$ is regularly varying with index $\alpha >0$ and non-null Radon measure $\mu$ on $\R^d_{\bZ}$ if and only if one of the following relations holds:
	\begin{enumerate}
		\item (Continuous form): The random variable $\lvert \bX \rvert$ is regularly varying with index $\alpha>0$ and $$\frac{\p{x^{-1} \bX \in \cdot}}{\p{\lvert \bX \rvert > x}} \convv \mu(\cdot), \text{ as } x \rightarrow \infty. $$
		\item (Weak convergence to independent radial/spectral decomposition): the following limit holds $$\p{\bigg(\frac{\bX}{x}, \frac{\bX}{\lvert \bX \rvert}  \bigg) \in \cdot} \cvgwk \p{(Y, \bTheta) \in \cdot}, \text{ as } x \rightarrow \infty$$ where $Y \sim \text{Pareto}(\alpha)$ with $\alpha >0$ and is independent of $\bTheta$, which takes values on the unit sphere $\S^{d-1}$ defined by $\S^{d-1} = \{\mathbf{x} \in \R^d: \lvert \mathbf{x}  \rvert = 1 \}$.
	\end{enumerate}
\end{prop} 

In Proposition~\eqref{prop:rvdecomp}, the notation $\convv$ refers to vague convergence: we say that a sequence of measures $(\mu_n)$ (with $\mu_n \in M_{+}(E)$, the space of nonnegative Radon measure on $(E, \cE)$) converges vaguely to a measure $\mu \in M_{+}(E)$ if for all functions $f \in \C_{K}^{+}(E)$, we have $\int_{E} f(x) \mu_n(\diff x) \rightarrow \int_{E} f(x) \mu(\diff x)$, where $C_K^{+}(E)$ denotes the set of functions $f: E \rightarrow \R_+$ being continuous with compact support. For more details about vague convergence, see e.g. Chapter 3 in \cite{resnick2}. The notation $\cvgwk$ refers to the standard notion of weak convergence. The above characterisations have various consequences. The first property is a continuous mapping theorem, first proved in \cite{hl06} in the framework of metric spaces. We use a simplified version fitting our settings, which we partially reproduce, from \cite{mw23}. See also Proposition 4.3 and Corollary 4.2 in \cite{lrr14}. 

\begin{prop}[Theorem 2.2.30 in \cite{mw23}, Proposition 4.3 and Corollary 4.2 in \cite{lrr14}]\label{prop:cm}
	Let $\bX$ be a random vector in $\R^d$ and suppose it is regularly varying with index $\alpha>0$ and non-null Radon measure $\mu$ on $\R^d_{\bZ}$. Let $g(\cdot): \R^d \rightarrow \R$ be a non-zero, continuous and positively homogeneous map of order $\gamma$, i.e. for every $\bx \in \R^d$, $g(t\bx)=t^{\gamma}g(\bx)$ for some $\gamma>0$. Then, the following limit relation holds $$\frac{\p{x^{-1} g(\bX) \in \cdot}}{\p{\lvert \bX \rvert^{\gamma} > x}} \convv \mu(g^{-1}(\cdot)), \text{ as } x \rightarrow \infty.$$ Note that for every $\epsilon>0$, $\mu(g^{-1}(\{x \in \R: \lvert x \rvert > \epsilon\})) < \infty. $ Moreover, if $\mu(g^{-1}(\cdot))$ is not the null measure on $\R_0$, then $g(\bX)$ is regularly varying with index $\alpha/\gamma$ and with non-null Radon measure $$\frac{\mu(g^{-1}(\cdot))}{\mu(g^{-1}(\{x \in \R: \lvert x \rvert > 1\}))}.$$  
\end{prop}

\begin{ex}\label{ex:proj}
	It is easily seen that the map defined by the projection on any coordinate of $\bX$ is a continuous mapping satisfying the assumptions of Proposition~\eqref{prop:cm} with $\gamma=1$. If $d=2$, $\bX=(X_1, X_2)$ and $g(\bX) \defeq X_1$, then by the homogeneity property of the limiting Radon measure $\mu$, as long as $$\mu(\{(x_1, x_2) \in \R^2_{\bZ}: x_1 > 1 \}) > 0$$ one obtains regular variation of $X_1$ with index $\alpha > 0$. 
\end{ex}

A second useful result, due to \cite{szm16} again in the setting of metric spaces that we simplify here, shows that one can actually replace the norm $\lvert \cdot \rvert$ by any modulus. A modulus, as defined in Definition 2.2 of \cite{szm16}, is a function $\rho: \R^d \rightarrow [0, \infty)$ such that $\rho(\cdot)$ is non-zero, continuous and positively homogeneous of order 1. Proposition 3.1 in \cite{szm16} then ensures the following. 

\begin{prop}[Proposition 3.1 in \cite{szm16}]\label{prop:mod}
	A random vector $\bX$ with values in $\R^d$ is regularly varying with index $\alpha>0$ and non-null Radon measure $\mu$ on $\R^d_{\bZ}$ if and only if there exists a modulus $\rho$ such that $\rho(\bX)$ is regularly varying with index $\alpha > 0$, and a random vector $\bTheta$ taking values on $\S^{d-1} \defeq \{\bx \in \R^d: \rho(\bx) = 1\}$ such that $$\Cprob{\frac{\bX}{\rho(\bX)} \in \cdot}{\rho(\bX) > x} \cvgwk \p{\bTheta \in \cdot}, \text{ as } x \rightarrow \infty.$$  
\end{prop}

Finally, in subsequent sections, we shall also use an other characterisation via the regular variation of linear combinations, proven by \cite{bdm02}. We denote the inner product in $\R^d$ by $\langle \cdot, \cdot \rangle$. 

\begin{prop}[Proposition 1.1 in \cite{bdm02}]\label{prop:lincomb}
	A random vector $\bX$ with values in $\R^d$ is regularly varying with noninteger index $\alpha>0$ if and only if there exists a slowly varying function $L(\cdot)$ such that, for all  $\bt \in \R^d$, $$\lim_{x \rightarrow \infty} \frac{\p{\langle \bt, \bX \rangle > x}}{x^{-\alpha} L(x)} = w(\bt) \text{ exists, }$$ for some function $w(\cdot)$ and there exists one $\bt_0 \neq 0$ such that $w(\bt_0) > 0.$ 
\end{prop}

The above result states that a random vector $\bX$ is regularly varying with index $\alpha>0$ if and only if all linear combinations of its components are regularly varying with the same index $\alpha>0$. Note that it is not necessary for $\alpha$ in Proposition~\eqref{prop:lincomb} to be noninteger for the above equivalence to hold; however, when this is not the case, there are some caveats that we avoid considering in our the results of upcoming sections (e.g., with $\alpha$ noninteger, we do not have to consider $t \in \R^{d}$ but rather $t \in \R^{d}_+$), see \cite{bdm02}. 

Finally, the last result of great importance in showing the transfer of regular variation in the subsequent sections is Karamata's Theorem, which can be found as Theorem 8.1.6 in \cite{bgt89}. Let $X$ be a random variable, denote its associated Laplace-Stieltjes transform by $\varphi_X(s) \defeq \E{e^{-sX}}$ for $s > 0$, and its $n$-th derivative by $\varphi_X^{(n)}(s) = \E{(-X)^n e^{-sX}}$. Let $\Gamma(\cdot)$ define the Gamma function. 

\begin{thm}[Karamata's Tauberian Theorem, Theorem 8.1.6 in \cite{bgt89}]\label{thm:kar}
	The following statements are equivalent: 
	\begin{enumerate}
		\item $X$ is regularly varying with noninteger index $\alpha>0$ and slowly varying function $L_X(\cdot)$, i.e. $$\p{X > x} \sim  x^{-\alpha}L_X(x), \text{ as } x \rightarrow \infty.$$
		\item For a noninteger index $\alpha > 0$,  $$\varphi_X^{(\lceil \alpha \rceil)}(s) \sim C_{\alpha} s^{\alpha-\lceil \alpha \rceil}L_X(1/s), \text{ as } s \rightarrow 0^+,$$ for $L_X(\cdot)$ a slowly varying function, where $C_{\alpha}:=-\Gamma(\alpha+1) \Gamma(1-\alpha)/\Gamma(\alpha - \lfloor \alpha \rfloor).$
	\end{enumerate}
\end{thm}
\begin{rem}\label{rem:kar}
	Note that when $X$ is regularly varying with index $\alpha \in (n, n+1)$, the $(n+1)$-th moment does not exist. Observe that the above trivially implies that, when $\alpha \in (n, n+1)$, $\varphi_X^{(n+1)}(s) = \varphi_X^{(\lceil \alpha \rceil)}(s) \rightarrow \infty$, as $s \rightarrow 0^+$, a property we will use repeatedly in subsequent sections. 	
\end{rem}

\section{Tail asymptotics of maximum functional in renewal Poisson cluster process}

We now prove a single big-jump principle for the tail asymptotics of the distribution of the maximum functional of a generic cluster in the settings of the renewal Poisson cluster process. As mentioned in Remark~\eqref{rem:simple-pro}, the conclusions reached for this process are of course valid for the mixed binomial Poisson cluster process. 

\begin{prop}\label{prop:max-rp-trsf}
	Suppose the vector $\left(X, K_A\right)$ in Equation~\eqref{eq:rp-h} is regularly varying with index $\alpha > 1$ and non-null Radon measure $\mu$. Then, $$\p{H^{R}>x} \sim (1+\E{K_A})\p{X>x}, \text{ as } x \rightarrow \infty.$$
	Moreover, if $\mu(\{(x_1, x_2) \in \R^2_{+, \bZ}: x_1 > 1\}) > 0$, then $H^{R}$ is regularly varying with index $\alpha>1$.   
\end{prop}
\begin{proof}[Proof of Proposition~\eqref{prop:max-rp-trsf}]
    The proof can be found in Appendix~\eqref{app:1}. It uses a classical approach via conditioning on $K_A$ and Taylor expansions and is given for completeness. 
\end{proof}

\begin{rem}
    In the proof of Proposition~\eqref{prop:max-rp-trsf}, one only needs $X$ to be regularly varying for $H^H$ to be regularly varying. However, to keep the same settings in terms of regular variation as for the upcoming results, we make the assumption that $(X, K_A)$ is regularly varying and regular variation of $X$ follows by considering the consequences of this assumption contained in Example~\eqref{ex:proj}. The case where $\p{X>x}=o(\p{K_A>x})$, $x\to \infty$, $X$ regularly varying and $K_A$ a stopping time with respect to $(A_j)_{j\ge0}$ is treated in Proposition 3.1 and Corollary 4.2 of \cite{bmz22}. It is proved that $H^{R}$ is also regularly varying but, more generally, that $H^R$ falls in the same MDA than $X$. What we propose in Proposition~\eqref{prop:max-rp-trsf} is merely a refinement for the Fr\'{e}chet MDA, describing explicitly the tail of $H^{R}$ when $\p{K_A>x}=\cO(\p{X>x})$, $x\to \infty$, and $K_A$ depending only on $X_0$.
\end{rem}

\section{Tail asymptotics of the sum functional in renewal Poisson cluster process}

We now prove a result concerning the sum functional of a generic cluster in the settings of the renewal Poisson cluster process. Again, this extends easily to the mixed binomial Poisson cluster process. 

\begin{prop}\label{prop:sum-rp-trsf}
    	Suppose the vector $\left(X, K_A\right)$ in Equation~\eqref{eq:rp-d} is regularly varying with noninteger index $\alpha > 1$. Then, $D^{R}$ is regularly varying with the same index $\alpha$. More specifically, $$\p{D^{R}>x} \sim \p{X + \E{X} K_A > x} + \E{K_A}\p{X>x}, \text{ as } x \rightarrow \infty.$$
\end{prop}

\begin{proof}[Proof of Proposition~\eqref{prop:sum-rp-trsf}]
 First, note that the Laplace-Stieltjes transform of $D^{R}$ in Equation~\eqref{eq:rp-d} is given by 
	\begin{align*}
		\varphi_{D^{R}}(s) \defeq \E{e^{-s X - s \sum_{j=1}^{K_A} X_j}}&= \E{\E{e^{-s X} e^{- s \sum_{j=1}^{K_A} X_j} \; \big\lvert \; A }}\\
  &= \E{e^{-sX} e^{K_A \log \E{e^{-sX}}}} \eqdef \E{e^{-sX + K_A \log \varphi_X(s)}}
	\end{align*}
upon recalling that $X \defeq f(A)$ and $K_A$ are independent conditionally on the ancestral mark $A$, and that $(X_j)_{j \geq 1}$ are i.i.d. and independent of $A$. 	
We first show that, for any noninteger $\alpha \in (n, n+1)$, $n \in \N$, $$\varphi_{D^{R}}^{(n+1)}(s) \sim \varphi_{X + \E{X}K_A}^{(n+1)}(s) + \E{K_A}\varphi_{X}^{(n+1)}(s), \text{ as } s \rightarrow 0^+,$$ where $\varphi_{X + \E{X}K_A}(s) \defeq \E{e^{-sX - s \E{X} K_A }},$ and where $\varphi_X^{(n)}$ is the $n$th derivative of the Laplace-Stieltjes transform of a random variable $X$.   	

We have to consider the following expression:
\begin{align}\label{eq:ordern}
	\bigg\lvert \varphi_{D^{R}}^{(n+1)}(s)-\big(\varphi_{X+\E{X}K_A}^{(n+1)}(s)+\E{K_A}\varphi_X^{(n+1)}(s)\big) \bigg\rvert &=	 \bigg\lvert \E{\bigg(-X+K_A \frac{\varphi_X^{(1)}(s)}{\varphi_X(s)}\bigg)^{n+1} e^{-sX + K_A \log \varphi_X(s)}} \nonumber \\
	&\quad + \E{K_A \frac{\varphi_X^{(n+1)}(s)}{\varphi_X(s)} e^{-sX + K_A \log \varphi_X(s)}} \nonumber \\
	&\quad - \E{\bigg(-X - \E{X}K_A\bigg)^{n+1}e^{-sX - s \E{X}K_A}} \nonumber \\
	&\quad - \E{K_A}\E{\big(-X\big)^{n+1}e^{-sX}} + C_{n+1} \bigg\rvert \nonumber \\
	&\eqdef \big\lvert B_1 + B_2 - B_3 - B_4 + C_{n+1} \big\rvert.
\end{align}

Consider first the difference $\big\lvert B_1 - B_3 \big\rvert.$ The following set of inequalities, directly due to the convexity of the function $\log \varphi_X(\cdot)$, will prove useful in controlling the above difference: for $s > 0$, we have
\begin{equation}\label{eq:convineq} 
	- s\E{X} K \leq K \log \varphi_X(s) \leq sK \frac{\varphi_X^{(1)}(s)}{\varphi_X(s)} \leq 0 \leq -sK \frac{\varphi_X^{(1)}(s)}{\varphi_X(s)} \leq -K \log \varphi_X(s) \leq s \E{X}K.
\end{equation}

Using the basic decomposition $(a^{n+1}-b^{n+1}) = (a-b) \sum_{k=0}^{n} a^{n-k}b^{k}$ as well as Equation~\eqref{eq:convineq} yields
\begin{align*}
	\big\lvert B_1 - B_3 \big\rvert &\leq \bigg\lvert \bigg(\frac{\varphi_X^{(1)}(s)}{\varphi_X(s)}+\E{X} \bigg) \\
 &\quad \cdot \E{K_A \bigg(\sum_{k=0}^n \bigg(-X+K_A \frac{\varphi_X^{(1)}(s)}{\varphi_X(s)}\bigg)^{n-k}\bigg(-X - \E{X} K_A \bigg)^{k} \bigg) e^{-sX+K_A \log \varphi_X(s)} } \bigg\rvert \\
	&\leq \bigg\lvert \bigg(\frac{\varphi_X^{(1)}(s)}{\varphi_X(s)}+\E{X} \bigg) \bigg( \E{K_A\bigg(-X + K_A \frac{\varphi_X^{(1)}(s)}{\varphi_X(s)}\bigg)^{n} e^{-sX + K_A \log \varphi_X(s)} } \\
	&\quad + \E{K_A\bigg(-X - \E{X} K_A\bigg)^{n} e^{-sX + K_A \log \varphi_X(s)} } \\
	&\quad + \E{K_A \bigg(\sum_{k=1}^{n-1} \bigg(-X+K_A \frac{\varphi_X^{(1)}(s)}{\varphi_X(s)}\bigg)^{n-k}\bigg(-X - \E{X} K_A \bigg)^{k} \bigg) e^{-sX+K_A \log \varphi_X(s)} } \bigg) \bigg\rvert \\
	&\eqdef \big\lvert G  (B_{11}+B_{12}+B_{13}) \big\rvert.
\end{align*} 
where $G \defeq \frac{\varphi_X^{(1)}(s)}{\varphi_X(s)} + \E{X}.$

We then treat each term separately. First, consider $B_{11}$. Using the binomial theorem, we have that $$\bigg(-X + K_A \frac{\varphi_X^{(1)}(s)}{\varphi_X(s)}\bigg)^n = \sum_{j=0}^n \binom{n}{j} \bigg(-X\bigg)^j \bigg(K_A \frac{\varphi_X^{(1)}(s)}{\varphi_X(s)}\bigg)^{n-j}.$$
Using the linearity of expectations, we separate the cases. Let $j=0$. Because $G > 0$, $K_A \frac{\varphi_X^{(1)}(s)}{\varphi_X(s)}<0$, using Equation~\eqref{eq:convineq} and the basic inequality $xe^{-x}\leq e^{-1}$, we get: 
\begin{align}
	\bigg\lvert G  \E{K_A \bigg(K_A \frac{\varphi_X^{(1)}(s)}{\varphi_X(s)} \bigg)^n e^{-sX + K_A \log \varphi_X(s)}} \bigg\rvert &\leq \frac{G}{s}  \E{K_A^{n-1} \bigg\lvert \bigg(\frac{\varphi_X^{(1)}(s)}{\varphi_X(s)} \bigg)^{n-1} \bigg\rvert \bigg(-K_A\log \varphi_X(s)\bigg)e^{K_A \log \varphi_X(s)}} \nonumber \\
	&\leq \frac{G}{s}  \E{K_A^{n-1} \bigg\lvert \bigg(\frac{\varphi_X^{(1)}(s)}{\varphi_X(s)} \bigg)^{n-1} \bigg\rvert e^{-1}}.\label{eq:tmp1} 
\end{align}
In order to control the upper bound, we need to control $G/s$, and we have to distinguish two cases: 
\begin{enumerate}
    \item[\textit{Case $\alpha \in (1,2)$:}] We have the identities
    \begin{align*}\frac{G}{s} &= \frac{\frac{\varphi_X^{(1)}(s)}{\varphi_X(s)} + \E{X}}{s} = \frac{\frac{\varphi_X^{(1)}(s)}{\varphi_X(s)} - \varphi_X^{(1)}(s)+ \varphi_X^{(1)}(s)+\E{X}}{s} = \varphi_X^{(1)}(s)\bigg(\frac{\frac{1}{\varphi_X(s)}-1}{s} \bigg) + \frac{\varphi_X^{(1)}(s)+\E{X}}{s}\,.
\end{align*}
The limit as $s \rightarrow 0^+$ of $\frac{\frac{1}{\varphi_X(s)}-1}{s}$ is the derivative of $1/\varphi_X(s)$ at $s=0$ and hence is finite; it follows that $$\varphi_X^{(1)}(s)\bigg(\frac{\frac{1}{\varphi_X(s)}-1}{s} \bigg) = \cO\big(\varphi_X^{(1)}(s) \big), \text{ as } s \rightarrow 0^+.$$
Now note that, for the second term, if first $X$ has negligible tails with respect to $X+\E{X}K_A$, by Lemma~\eqref{lem:rvcomp}, it follows that  $$\frac{\varphi_X^{(1)}(s)+\E{X}}{s} = o\big(\varphi_{X+\E{X}K_A}^{(2)}(s)+\E{K_A}\varphi_X^{(2)}(s)\big), \text{ as } s \rightarrow 0^+.$$ 
If $X$ is regularly varying with the same index as $X+\E{X}K_A$, then clearly, by adapting the proof of Lemma~\eqref{lem:rvcomp}, it follows that $$\frac{\varphi_X^{(1)}(s)+\E{X}}{s} = \cO\big(\varphi_{X+\E{X}K_A}^{(2)}(s)+\E{K_A}\varphi_X^{(2)}(s)\big), \text{ as } s \rightarrow 0^+.$$
By a dominated convergence argument, the upper bound in Equation~\eqref{eq:tmp1} is such that  $$\E{K_A^{n-1} \bigg\lvert \bigg(\frac{\varphi_X^{(1)}(s)}{\varphi_X(s)} \bigg)^{n-1} \bigg\rvert \bigg(-K_A\log \varphi_X(s)\bigg)e^{K_A \log \varphi_X(s)}}= o(1), \text{ as } s \rightarrow 0^+,$$ and combining with the arguments above, this proves that, no matter if $X$ is lighter or as heavy as the modulus $X+\E{X}K_A$, $$B_{11} = o\big(\varphi^{(2)}_{X + \E{X}K_A}(s)+ \E{K_A}\varphi_X^{(2)}(s)\big), \text{ as } s \rightarrow 0^+.$$ 
\item[\textit{Case $\alpha \in (n,n+1), \text{ with } n \in \N \backslash \{ 1 \}$:}] Using the definition of the derivative, as $s \rightarrow 0^+$, $$\lim_{s \rightarrow 0^+} \frac{G}{s} = \lim_{s \rightarrow 0^+} \frac{\frac{\varphi_X^{(1)}(s)}{\varphi_X(s)}+\E{X} }{s} \text{ and } \lim_{s \rightarrow 0^+} \frac{\frac{\frac{\varphi_X^{(1)}(s)}{\varphi_X(s)}+\E{X} }{s}}{\frac{\varphi_X^{(2)}(s)}{\varphi_X(s)}-\frac{(\varphi_X^{(1)}(s))^2}{(\varphi_X(s))^2}} = 1, $$ and, for this range of $\alpha \in (n, n+1)$ with $n \in \N \backslash \{1 \}$ $$\frac{\varphi_X^{(2)}(s)}{\varphi_X(s)}-\frac{(\varphi_X^{(1)}(s))^2}{(\varphi_X(s))^2} < \infty $$
which is finite since $\alpha \in (n, n+1)$ for $n \geq 2$. Because $\varphi_X^{(1)}(s)$ is finite,   Equation~\eqref{eq:tmp1} is finite. Upon applying Theorem~\eqref{thm:kar}, it follows that, as $s \rightarrow 0^+$, $$\bigg\lvert G  \E{K_A \bigg(K_A \frac{\varphi_X^{(1)}(s)}{\varphi_X(s)} \bigg)^n e^{-sX + K_A \log \varphi_X(s)}} \bigg\rvert = o\big(\varphi_{X+\E{X}K_A}^{(n+1)}(s)+\E{K_A}\varphi_X^{(n+1)}(s)\big).$$
\end{enumerate}

The treatment of terms where $j>0$ is easier: it is sufficient to note that, whenever $X$ appears in the product, one can always ``lose a power": suppose without loss of generality that $j=1$ in the decomposition due to the binomial theorem above; we are left to consider the following term
$$\bigg\lvert G  \E{K_A \bigg\{ \binom{n}{1} \bigg(-X\bigg)^1 \bigg(K_A \frac{\varphi_X^{(1)}(s)}{\varphi_X(s)}\bigg)^{n-1} \bigg\} e^{-sX + K_A \log \varphi_X(s)} } \bigg\rvert.$$
This is smaller than 
$$ \frac{G}{s}  \E{\binom{n}{1}K_A^n  \bigg\lvert \bigg(\frac{\varphi_X^{(1)}(s)}{\varphi_X(s)}\bigg)^{n-1}\bigg\rvert  \big(sX\big) e^{-sX}} \leq \frac{G}{s}  \E{\binom{n}{1}K_A^n  \bigg\lvert \bigg(\frac{\varphi_X^{(1)}(s)}{\varphi_X(s)}\bigg)^{n-1}\bigg\rvert  e^{-1}},  $$
and by similar reasoning as above, the expectation as well as the whole of the upper bound is finite. All in all, this shows that, as $s \rightarrow 0^+$, 
\begin{multline*}
    \bigg\lvert G  \E{K_A \bigg\{ \binom{n}{1} \bigg(-X\bigg)^1 \bigg(K_A \frac{\varphi_X^{(1)}(s)}{\varphi_X(s)}\bigg)^{n-1} \bigg\} e^{-sX + K_A \log \varphi_X(s)} } \bigg\rvert  
 \\ = o\big(\varphi_{X+\E{X}K_A}^{(n+1)}(s)+\E{K_A}\varphi_X^{(n+1)}(s)\big)
\end{multline*}
Upon applying the same arguments on all terms making up $B_{11}$, using at times H\"{o}lder's inequality to justify that expectations of the form $\E{K_A \big(-X\big)^{j-1} \big(K_A \frac{\varphi_X^{(1)}(s)}{\varphi_X(s)}\big)^{n-j}}$ for $2 \leq j \leq n-1$ are finite, and one $X$ is factorised as in the reasoning above, this is sufficient to show that $$\big\lvert G  B_{11} \big\rvert = o\big(\varphi_{X+\E{X}K_A}^{(n+1)}(s)+\E{K_A}\varphi_X^{(n+1)}(s)\big), \text{ as } s \rightarrow 0^+.$$

A completely analogous approach - omitted for brevity - shows that $$\big\lvert G  B_{12} \big\rvert = o\big(\varphi_{X+\E{X}K_A}^{(n+1)}(s)+\E{K_A}\varphi_X^{(n+1)}(s)\big), \text{ as } s \rightarrow 0^+$$
replacing only the appeal to Equation~\eqref{eq:convineq} by the fact that we can always find $s>0$ small enough such that $s\E{X}  \leq -2\log \varphi_X(s)$, which holds because of the following reasoning: since $\varphi_X(s)$ is differentiable at 0, by the integrability of $X$, one obtains $$\lim_{s \rightarrow 0^+} -\frac{\log \varphi_X(s)}{s} = -\frac{\varphi_X^{(1)}(0)}{\varphi_X(0)} = -\E{-X} \iff \lim_{s \rightarrow 0^+} -\frac{\log \varphi_X(s)}{s} = \E{X}.$$ By a similar argument, $-\frac{2\log \varphi_X(s)}{s} \rightarrow 2 \E{X}, \text{ as } s \rightarrow 0^+$. Hence, there exists $s > 0$ small enough such that $$s\E{X} \leq -2 \log \varphi_X(s).$$ 

Finally, consider $\big\lvert G  B_{13} \big\rvert$. The sum given can be factorised as 
\begin{multline*}
    \sum_{k=1}^{n-1} \bigg(-X+K_A \frac{\varphi_X^{(1)}(s)}{\varphi_X(s)}\bigg)^{n-k}\bigg(-X - \E{X} K_A \bigg)^{k} \\ = \bigg(-X + K_A \frac{\varphi_X^{(1)}(s)}{\varphi_X(s)}\bigg) \sum_{k=1}^{n-1} \bigg(-X+K_A \frac{\varphi_X^{(1)}(s)}{\varphi_X(s)}\bigg)^{n-1-k}\bigg(-X - \E{X} K_A \bigg)^{k}
\end{multline*}

Now this yields, upon using Equation~\eqref{eq:convineq} and the basic inequality $x e^{-x} \leq e^{-1}$ in the last step, 
\begin{align*}
	\lvert G  B_{13} \rvert &= \frac{G}{s}  \mathbb{E}\bigg[K_A\bigg\lvert \bigg(-sX + K_A s \frac{\varphi_X^{(1)}(s)}{\varphi_X(s)}\bigg) \bigg\rvert e^{-\big(sX - K_A \log \varphi_X(s)\big)} \\ 
 &\qquad \cdot \sum_{k=1}^{n-1} \bigg\lvert\bigg(-X + K_A \frac{\varphi_X^{(1)}(s)}{\varphi_X(s)}\bigg)^{n-k-1} \bigg(-X - \E{X}K_A \bigg)^{k}\bigg\rvert\bigg] \\
	&\leq \frac{G}{s}  \mathbb{E}\bigg[K_A\bigg(sX - K_A \log \varphi_X(s) \bigg) e^{-\big(sX - K_A \log \varphi_X(s)\big)} \\
 &\qquad \cdot \sum_{k=1}^{n-1} \bigg\lvert\bigg(-X + K_A \frac{\varphi_X^{(1)}(s)}{\varphi_X(s)}\bigg)^{n-k-1} \bigg(-X - \E{X}K_A \bigg)^{k}\bigg\rvert \bigg] \\
	&\leq \frac{G}{s}  \E{K_A e^{-1} \sum_{k=1}^{n-1} \bigg\lvert\bigg(-X + K_A \frac{\varphi_X^{(1)}(s)}{\varphi_X(s)}\bigg)^{n-k-1} \bigg(-X - \E{X}K_A \bigg)^{k}\bigg\rvert}.
\end{align*}
The highest order of the product of the summands above is of power $n$: again, since $\alpha \in (n, n+1)$, using H\"{o}lder's inequality, the expectation is finite. Overall, this shows once again that $$\big\lvert G  B_{13} \big\rvert = o\big(\varphi_{X+\E{X}K_A}^{(n+1)}(s)+\E{K_A}\varphi_X^{(n+1)}(s)\big), \text{ as } s \rightarrow 0^+.$$
Collecting all of the above bounds, this shows that $$\big\lvert B_1 - B_3 \big\rvert = o\big(\varphi_{X+\E{X}K_A}^{(n+1)}(s)+\E{K_A}\varphi_X^{(n+1)}(s)\big), \text{ as } s \rightarrow 0^+.$$

Consider the difference $\big\lvert B_2 - B_4 \big\rvert,$ and note that one can write it as 
\begin{align*}
	\big\lvert B_2 - B_4 \big\rvert &= \bigg\lvert \frac{\varphi_X^{(n+1)}(s)}{\varphi_X(s)} \E{K_A \big(1-e^{-sX + K_A \log \varphi_X(s)}\big)} + \varphi_X^{(n+1)}(s)\bigg(\frac{1}{\varphi_X(s)} - 1 \bigg) \E{K_A} \bigg\rvert \\
	&\eqdef \big\lvert B_{21}+B_{22} \big\rvert.
\end{align*}
Now, by a dominated convergence argument as before, one has that $\E{K_A \big(1-e^{-sX + K_A \log \varphi_X(s)}\big)}=o(1)$, as $s \rightarrow 0^+,$ and hence, that $$\big\lvert B_{21} \big\rvert = o\big(\varphi_X^{(n+1)}(s)\big) = o\big(\varphi_{X+\E{X}K_A}^{(n+1)}(s)+\E{K_A}\varphi_X^{(n+1)}(s)\big), \text{ as } s \rightarrow 0^+. $$ Similarly, by the integrability of $K_A$, $$\big\lvert B_{22} \big\rvert = o\big(\varphi_X^{(n+1)}(s)\big) = o\big(\varphi_{X+\E{X}K_A}^{(n+1)}(s)+\E{K_A}\varphi_X^{(n+1)}(s)\big), \text{ as } s \rightarrow 0^+.$$ Collecting the above, this implies that $$\big\lvert B_2 - B_4 \big\rvert = o\big(\varphi_{X+\E{X}K_A}^{(n+1)}(s)+\E{K_A}\varphi_X^{(n+1)}(s)\big), \text{ as } s \rightarrow 0^+.$$
Lastly, the terms making up $C_{n+1}$ when $\alpha \in (n,n+1)$ are all the terms (and cross-products) of order strictly lower than $n+1$ and, consequently, are finite. It follows by Theorem~\eqref{thm:kar} that  $$C_{n+1} = o\big(\varphi_{X+\E{X}K_A}^{(n+1)}(s)+\E{K_A}\varphi_X^{(n+1)}(s)\big), \text{ as } s \rightarrow 0^+.$$
All in all, this essentially shows that, as $s \rightarrow 0^+$,
$$\bigg\lvert \varphi_{D^{R}}^{(n+1)}(s)-\big(\varphi_{X+\E{X}K_A}^{(n+1)}(s)+\E{K_A}\varphi_X^{(n+1)}(s)\big) \bigg\rvert = o\big(\varphi_{X+\E{X}K_A}^{(n+1)}(s)+\E{K_A}\varphi_X^{(n+1)}(s)\big),$$ and hence that $$\varphi_{D^{R}}^{(n+1)}(s) \sim \varphi_{X+\E{X}K_A}^{(n+1)}(s)+\E{K_A}\varphi_X^{(n+1)}(s), \text{ as } s \rightarrow 0^+,$$ and this equivalence holds for any $\alpha \in (n, n+1)$, $n \in \N$. 

Because the modulus $X + \E{X}K_A$ is regularly varying whenever $(X, K_A)$ is - see Remark~\eqref{rem:modulus} - Karamata's Theorem~\eqref{thm:kar} implies that $$\varphi_{X+\E{X}K_A}^{(n+1)}(s) \sim C_{\alpha} s^{\alpha-\lceil \alpha \rceil}L_{X + \E{X}K_A}(1/s), \text{ as } s \rightarrow 0^+$$ for some slowly varying function $L_{X + \E{X}K_A}(\cdot)$. 
Then, suppose first that $X$ is not regularly varying and has negligible tails with respect to the modulus $X + \E{X}K_A$. 
Then Lemma~\eqref{lem:negli} yields that $$\varphi_X^{(n+1)}(s) = o\big(\varphi_{X+ \E{X}K_A}^{(n+1)}(s)\big), \text{ as } s \rightarrow 0^+$$ and hence, this implies that $$\varphi_{D^R}^{(n+1)}(s) \sim C_{\alpha}s^{\alpha-\lceil \alpha \rceil}L_{X+\E{X}K_A}(1/s)\big(1+o(1)\big), \text{ as } s \rightarrow 0^+$$ which yields by re-applying Karamata's Tauberian Theorem~\eqref{thm:kar}, that $$\p{D^R > x} \sim x^{-\alpha} L_{D^R}(x) \sim x^{-\alpha}L_{X+\E{X}K_A}(x)\big(1+o(1)\big), \text{ as } x \rightarrow \infty.$$

In the case where $X$ is regularly varying, by Example~\eqref{ex:proj}, and because $X$ has the same index $\alpha > 1$ as the modulus $X+\E{X}K_A$, if the limiting Radon measure is non-null on the correct subspace,  Karamata's Tauberian Theorem~\eqref{thm:kar} yields, $$\varphi_X^{(n+1)}(s) \sim C_{\alpha} s^{\alpha-\lceil \alpha \rceil} L_{X}(1/s), \text{ as } s \rightarrow 0^+.$$ Then, for each $n \in \N$ 
\begin{align*}
    \varphi_{D^{R}}^{(n+1)}(s) \sim 2C_{\alpha}s^{\alpha-\lceil \alpha \rceil}\big(L_{X + \E{X}K_A}(1/s) + \E{K_A} L_{X}(1/s) \big), \text{ as } s \rightarrow 0^+,
\end{align*}
and because the sum of two slowly varying function is still a slowly varying function, \linebreak $L_D^R(\cdot) \defeq L_{X + \E{X}K_A}(\cdot) + \E{K_A}L_{X}(\cdot)$ is slowly varying. Applying again Karamata's Tauberian Theorem~\eqref{thm:kar} in the other direction, yields $$\p{D^R > x} \sim x^{-\alpha} L_{D^R}(x) \sim x^{-\alpha} \big(L_{X + \E{X}K_A}(x) + \E{K_A}L_{X}(x) \big),  \text{ as } x \rightarrow \infty$$ which yields the desired result and the proof is complete. 
\end{proof}

\begin{rem}\label{rem:modulus}
	Note that the assumption that the random vector $(X, K_A)$ is regularly varying with index $\alpha>1$ ensures, by Proposition~\eqref{prop:mod}, that $X+\E{X}K_A$ is regularly varying with the same index $\alpha>1$. Indeed, it can be easily seen that $\rho(X, K_A) \defeq X+\E{X}K_A$ is a modulus (in the sense made precise in Section 3), provided that $\E{X} \neq 0$, which is a natural assumption to make, since $X$ is taken to be nonnegative.  
\end{rem}

\begin{rem}



	Note that the findings of Proposition~\eqref{prop:sum-rp-trsf} are consistent with the findings of \cite{fgams06}: in particular, if $X$ and $K_A$ are independent - which is the setting in the aforementioned paper - or even if $X$ and $K_A$ are asymptotically independent (i.e. if $\p{X > x, \E{X}K_A > x} = o\big(\p{X > x}  \p{\E{X}K_A > x}\big)$, as $x \rightarrow \infty$) then the proposed asymptotics of Proposition~\eqref{prop:sum-rp-trsf} encompass three cases, depending on the relation between $X$ and $K_A$:
 \begin{enumerate}
     \item when  $\p{K_A> x} = o(\p{X > x}), \text{ as } x \rightarrow \infty$, then $\p{X + \E{X}K_A > x} \sim \p{X > x}$, as $x \rightarrow \infty.$ From Proposition~\eqref{prop:sum-rp-trsf}, this means that $$\p{D^R > x} \sim \p{X > x} + \E{K_A} \p{X > x} \sim (\E{K_A}+1) \p{X>x}, \text{ as } x \rightarrow \infty$$ which is equivalent to Proposition 4.1 in \cite{fgams06};
     \item when $\p{X > x} = o(\p{K_A> x}), \text{ as } x \rightarrow \infty$, then $\p{X + \E{X}K_A > x} \sim (\E{X})^{\alpha}\p{K_A> x}$, as $x \rightarrow \infty$. From Proposition~\eqref{prop:sum-rp-trsf}, this means that $$\p{D^R > x} \sim (\E{X})^{\alpha}\p{K_A> x}, \text{ as } x \rightarrow \infty$$ which is equivalent to Proposition 4.3 in \cite{fgams06};
     \item lastly, when $\p{K_A> x} \sim c\p{X > x}$, as $x \rightarrow \infty$, for $c > 0$, then $\p{X + \E{X}K_A > x} \sim \p{X>x} + c (\E{X})^{\alpha} \p{X>x}, \text{ as } x \rightarrow \infty$.   From Proposition~\eqref{prop:sum-rp-trsf}, this means that $$\p{D^R > x} \sim \big(\E{K_A}+1+c(\E{X})^{-\alpha}\big)\p{X > x}, \text{ as } x \rightarrow \infty$$ which is equivalent to Lemma 4.7 in \cite{fgams06}. 
 \end{enumerate}
 
Our approach offers a more flexible framework for dependence between the governing components of the clusters, namely $X$ and $K_A$. Yet, in this latter direction, and more closely related to our results, \cite{oc21} shows in a recent contribution that $$\p{D^R > x} \sim \I{\E{K_A} < \infty } \E{K_A} \p{X > x} + \p{\E{X} K_A > x}, \text{ as } x \rightarrow \infty,  $$ in the regime where $(X, K_A)$ are arbitrarily dependent and either $K_A$ is intermediate regularly varying and $\p{X > x} = o(\p{K_A > x}), \text{ as } x \rightarrow \infty$ (Theorem 6.10 in \cite{oc21}) or $X$ is intermediate regularly varying and $\p{K_A > x} = o(\p{X > x}), \text{ as } x \rightarrow \infty$ (Theorem 6.11 in \cite{oc21}). The novelty in this paper is to propose similar asymptotics in the case where $K_A$ and $X$ are effectively tail equivalent. 

\end{rem}

Note that the content of Proposition~\eqref{prop:sum-rp-trsf} is a kind of "double" big-jump principle: the heavy-tailedness introduced by letting the vector $(X, K_A)$ be regularly varying implies that there is two ways for the sum $D^R$ to be large; either through a combination of the dependent variables $X$ and $K_A$ or through the classical single big-jump coming from the additional term $\E{K_A}\p{X > x}$ consisting of the offspring events. 

\section{Tail asymptotics of the maximum functional in the Hawkes process}

We now propose a single big-jump principle concerning the maximum functional of a generic cluster in the settings of the Hawkes process. Recall that $\E{L_A} = \E{\kappa_A} = 1.$

\begin{prop}\label{prop:max-hawkes-trsf}
	Suppose the vector $(X, \kappa_A)$ in Equation~\eqref{eq:h-h} is regularly varying with index $\alpha > 1$ and non-null Radon measure $\mu$. Then, $$\p{H^{H}>x} \sim \frac{1}{1-\E{\kappa_A}}\p{X>x}, \text{ as } x \rightarrow \infty.$$
Moreover, if $\mu(\{(x_1, x_2) \in \R^2_{+, \bZ}: x_1 > 1\}) > 0$, then $H^{H}$ is regularly varying with index $\alpha>1$. 
\end{prop}

\begin{proof}[Proof of Proposition~\eqref{prop:max-hawkes-trsf}]
    The proof can be found in Appendix~\eqref{app:2} and follows the same approach as the proof of Proposition~\eqref{prop:max-rp-trsf}.
\end{proof}

\begin{rem}
    As hinted in Section~\ref{sct:introduction}, a closely related work concerning the maxima of the marks in a generic cluster of the Hawkes process can be found in \cite{bmz22}. Under the assumption that $K_A$ is a stopping time with respect to a filtration including the information about $(X_{ij})$, it is shown in their Lemma 4.1 that $H^H$ falls in the same MDA as $X$. What we propose in Proposition~\eqref{prop:max-hawkes-trsf} is merely a refinement for the Fr\'{e}chet MDA, describing explicitly the tail of $H^H$.  
\end{rem} 

\section{Tail asymptotics of the sum functional in Hawkes process}

We now propose another "double" big-jump principle concerning the sum functional of a generic cluster in the setting of the Hawkes process. The tail approximation obtained in Proposition~\eqref{prop:sum-hawkes-trsf} below is in fact very similar to the one in Proposition~\eqref{prop:sum-rp-trsf}, where both a single big-jump principle and a combination of the effects of the dependent variables $X$ and $\kappa_A$ yield large values for $D^H$. 

\begin{prop}\label{prop:sum-hawkes-trsf}
	Assume that $(X, \kappa_A)$ in Equation~\eqref{eq:h-d} has a regularly varying distribution with noninteger index $\alpha>1$. Then, $(X, L_A)$ is regularly varying with the same index $\alpha$. Further, $D^{H}$ is regularly varying with index $\alpha$. In fact, $$\p{D^{H}>x} \sim \frac{1}{1-\E{\kappa_A}} \p{X+\bigg(\frac{\E{X}}{1-\E{\kappa_A}}\bigg)\kappa_A > x}, \text{ as } x \rightarrow \infty. $$ 
\end{prop}

\begin{proof}[Proof of Proposition~\eqref{prop:sum-hawkes-trsf}]
Recall that the assumption that $(X, \kappa_A)$ is regularly varying with index $\alpha>1$ is equivalent to the regular variation of the linear combinations $t_1 X + t_2 \kappa_A$ for all $t_1, t_2 \in \R_+$ by Proposition~\eqref{prop:lincomb}. Similarly as in the proof of Proposition~\eqref{prop:sum-rp-trsf}, if we can show, at any order $(n+1)$ for $n \in \N$, and for any $t_1, t_2 \in \R_+$, that the behaviour of $\varphi_{t_1X+t_2 \kappa_A}^{(n+1)}(s) \defeq \frac{\partial^{n+1}}{\partial s^{n+1}} \big(\E{e^{-s(t_1 X + t_2 \kappa_A)}}\big),$ and that of $\varphi_{t_1X+t_2 L_A}^{(n+1)}(s) \defeq \frac{\partial^{n+1}}{\partial s^{n+1}} \big(\E{e^{-s(t_1 X + t_2 L_A)}}\big)$, as $s \rightarrow 0^+$ are comparable, i.e. if $$\varphi_{t_1X+t_2 \kappa_A}^{(n+1)}(s) \sim \varphi_{t_1X+t_2 L_A}^{(n+1)}(s), \text{ as } s \rightarrow 0^+,$$ then by Karamata's Theorem~\eqref{thm:kar}, we have $$ \p{t_1 X + t_2 \kappa_A > x} \sim \p{t_1 X + t_2 L_A > x}, \text{ as } x \rightarrow \infty.$$ But this essentially means, reapplying Proposition~\eqref{prop:lincomb}, that $(X, L_A)$ is regularly varying.

The following bounds will be useful: 
\begin{enumerate}
    \item By a Taylor expansion, as $s \rightarrow 0^+$, \begin{equation}\label{eq:hawkesb1} s-(1-e^{-s}) \leq s^2/2. \end{equation}
    \item For $s >0$ small enough, \begin{equation}\label{eq:hawkesb2}-(1-e^{-s}) \leq -s/2.\end{equation}
\end{enumerate}

First, note that it is possible to write $\varphi_{t_1X+t_2 L_A}(\cdot)$ as a function of $\kappa_A$ instead of $L_A$. Using the Tower property and recalling that $L_A \lvert A \sim \text{Poisson}(\kappa_A)$ yields
\begin{align*}
\varphi_{t_1X+t_2 L_A}(s)&=\E{\E{e^{-st_1 X - st_2 L_A}\; \vert \; A}}  =\E{e^{-st_1 X - (1-e^{-st_2})\kappa_A}}.
\end{align*}

From this, and letting $\alpha \in (n,n+1)$, $n \geq 1$, simple derivations and collection of terms lead us to consider the difference given by
\begin{align*}
    \big\lvert \varphi_{t_1 X + t_2 L_A}^{(n+1)}(s)-\varphi_{t_1 X + t_2 \kappa_A}^{(n+1)}(s)  \big\rvert &= \bigg\lvert \E{(-t_1 X)^{n+1} \big(e^{-st_1 X -(1-e^{-st_2})\kappa_A}-e^{-st_1 X -st_2 \kappa_A}\big)} \\
    &\quad + I_1 \E{(-t_1 X)^{n}(-t_2 \kappa_A) \big(e^{-st_1 X - (1-e^{-st_2} \kappa_A)\kappa_A-I_2 st_2} - e^{-st_1 X - st_2 \kappa_A} \big) } \\
    &\quad + ... \\
    &\quad + I_{j} \E{(-t_1 X)(-t_2 \kappa_A)^{n} \big(e^{-st_1 X - (1-e^{-st_2} \kappa_A)\kappa_A-I_k st_2} - e^{-st_1 X - st_2 \kappa_A} \big) } \\
    &\quad + \E{(-t_2 \kappa_A)^{n+1} \big(e^{-st_1 X -(1-e^{-st_2})\kappa_A- (n+1) st_2}-e^{-st_1 X -st_2 \kappa_A}\big)} + C_{n+1} \bigg\rvert \\
    &\eqdef \big\lvert B_1 + B_{21} + ... + B_{2j} + B_3 + C_{n+1} \big\rvert 
\end{align*}
where the constants of product terms $(B_{21},\ldots, B_{2j})$ $I_1, I_2, \ldots, I_j, I_k \in \N$ depend on $n$. 

Consider term $B_1$. Using Equation~\eqref{eq:hawkesb1} and Equation~\eqref{eq:hawkesb2} and the basic inequality $xe^{-x} \leq e^{-1}$, one can show that
\begin{align*}
\lvert B_1 \rvert &= \E{(t_1 X)^{n+1} e^{-st_1X} \big(e^{-(1-e^{-st_2})\kappa_A}-e^{-st_2 \kappa_A}\big)} \\
&\leq \E{(t_1 X)^{n+1} e^{-st_1X} e^{-(1-e^{-st_2})\kappa_A} \kappa_A (st_2 -1 + e^{-st_2})}  \\
&\leq \E{(t_1X)^{n+1} e^{-st_1 X} e^{-st_2 \kappa_A / 2} \kappa_A (st_2)^2/2} \\
&\leq \E{(st_1Xe^{-st_1 X}) \bigg(\frac{st_2 \kappa_A}{2} e^{-st_2 \kappa_A/2}\bigg) t_2 (t_1X)^{n}} \\ 
&\leq \E{e^{-2} t_2 (t_1 X)^{n}}.
\end{align*}
and by the finiteness of the $n$th moment of $X$ when $\alpha \in (n, n+1)$, the above expectation is finite. Hence, it follows, using Karamata's Theorem~\eqref{thm:kar} and Remark~\eqref{rem:kar}, that $$B_1 = o\big(\varphi_{t_1 X + t_2 \kappa_A}^{(n+1)}(s)\big), \text{ as }s \rightarrow 0^+.$$
Consider one representative for the cross-product terms, say, without loss of generality, $B_{21}$. Then, proceeding as before for term $B_1$, using Equation~\eqref{eq:hawkesb1} and Equation~\eqref{eq:hawkesb2} and the basic inequality $xe^{-x} \leq e^{-1}$, yields 
\begin{align*}
\lvert B_{21} \rvert &= I_1\E{(t_1 X)^{n} (t_2 \kappa_A) e^{-st_1 X}\big((e^{-(1-e^{-st_2})\kappa_A -st_2}-e^{-st_2 \kappa_A -st_1}) - (e^{-st_2 \kappa_A}-e^{-st_2 \kappa_A-st_1})\big)} \\
&\leq I_1 \E{(t_1 X)^{n}e^{-st_1 X} (t_2 \kappa_A^2) e^{-(1-e^{-st_2})\kappa_A}(st_2 - 1 + e^{-st_2})} \\
&\leq I_1 \E{(t_1 X)^{n}e^{-st_1 X} (t_2 \kappa_A^2) e^{-st_2 \kappa_A / 2} \frac{(st_2)^2}{2}} \\
&\leq I_1\E{(st_1 X e^{-st_1X}) \bigg(\frac{st_2 \kappa_A}{2}e^{-st_2 \kappa_A/2}\bigg) (t_1 X)^{n-1}  t_2^2 \kappa_A} \\
&\leq I_1 \E{e^{-2} (t_1 X)^{n-1} t_2^2 \kappa_A}. 
\end{align*}
Using H\"{o}lder's inequality, because the order of the product of $X^{n-1}$ and $\kappa_A$ is $n$, one obtains that the above expectation is finite. It follows from Karamata's Theorem~\eqref{thm:kar} and Remark~\eqref{rem:kar}, that 
$$B_{21} = o\big(\varphi_{t_1 X + t_2 \kappa_A}^{(n+1)}(s)\big), \text{ as }s \rightarrow 0^+,$$ and similarly for each cross product term $B_{22}, \ldots, B_{2j}$.  

Consider now $B_3$. With similar tools as before, using Equation~\eqref{eq:hawkesb1} and Equation~\eqref{eq:hawkesb2} and the basic inequality $x^2e^{-x} \leq 4 e^{-2}$, yields
\begin{align*}
\lvert B_3 \rvert &= \E{(t_2 \kappa_A)^{n+1} e^{-st_1 X} \big((e^{-(1-e^{-st_2})\kappa_A-2st_2}-e^{st_2 \kappa_A - 2st_2})-(e^{-st_2 \kappa_A}-e^{-st_2 \kappa_A -2st_2})\big)} \\
&\leq \E{(t_2 \kappa_A)^{n+1} e^{-st_1 X} (e^{-(1-e^{-st_2})\kappa_A-2st_2}-e^{st_2 \kappa_A - 2st_2})} \\
&\leq \E{(t_2\kappa_A)^{n+1} e^{-(1-e^{-st_2})\kappa_A} (st_2 \kappa_A -(1-e^{-st_2})\kappa_A)}  \\
&\leq 2\E{t_2^{n+1} \kappa_A^{n+2} (st_2/2)^2e^{-st_2 \kappa_A/2} }\\
&\leq 2\E{(t_2 \kappa_A)^n 4e^{-2}},
\end{align*}
which essentially shows once again, using Karamata's Theorem~\eqref{thm:kar} and Remark~\eqref{rem:kar}, that 
$$B_{3} = o\big(\varphi_{t_1 X + t_2 \kappa_A}^{(n+1)}(s)\big), \text{ as }s \rightarrow 0^+.$$

Lastly, making up the remainder $C_{n+1}$ are terms of strictly smaller order than $n+1$. These are finite and trivially, using Karamata's Theorem~\eqref{thm:kar} and Remark~\eqref{rem:kar},  $$C_{n+1} = o\big(\varphi_{t_1 X + t_2 \kappa_A}^{(n+1)}(s)\big), \text{ as }s \rightarrow 0^+.$$ 

Collecting all of the above results, it follows that 
$$\big\lvert \varphi_{t_1 X + t_2 L_A}^{(n+1)}(s)-\varphi_{t_1 X + t_2 \kappa_A}^{(n+1)}(s) \big\rvert = o(\varphi_{t_1 X + t_2 \kappa_A}^{(n+1)}(s)), \text{ as } s \rightarrow 0^+$$ which essentially means that, for all $n \in \N$ $$\varphi_{t_1 X + t_2 L_A}^{(n+1)}(s) \sim \varphi_{t_1 X + t_2 \kappa_A}^{(n+1)}(s), \text{ as } s \rightarrow 0^+.$$

Now, by Karamata's Theorem~\eqref{thm:kar}, this means that, for all $t_1, t_2 \in \R_+$ $$\p{t_1 X + t_2 L_A > x} \sim  \p{t_1 X + t_2 \kappa_A > x}, \text{ as } x \rightarrow \infty, $$ and using Proposition~\eqref{prop:lincomb}, this means that $(X, L_A)$ is regularly varying with index $\alpha>1$. We conclude by applying Theorem 1 in \cite{af18} which yields the desired result.
\end{proof}

\begin{rem}
    Proposition~\eqref{prop:sum-hawkes-trsf} is essentially about showing that if $(X, \kappa_A)$ is regularly varying, then $(X, L_A)$ is also regularly varying, furthermore with the same index $\alpha>1$. The equivalence between the regularly varying property of $\kappa_A$ and that of $L_A$ is easy to prove and is to be found, for example, in \cite{lsv07}. The crucial step to obtain the tail asymptotic of $D^{H}$ and its regularly varying property in Proposition~\eqref{prop:sum-hawkes-trsf} relies on Theorem 1 in \cite{af18}. In their even more general setting, the distribution of $X+c L_A$ is intermediate regularly varying, for all $c \in (\E{D^H} - \epsilon, \E{D^H} + \epsilon)$ for some $\epsilon>0$: this assumption encompasses the case where $(X, L_A)$ is regularly varying, but also the cases where $X$ (respectively $L_A$) is intermediate regularly varying and $L_A$ (respectively $X$) is lighter, in the sense that $\p{L_A > x} = o(\p{X > x}), \text{ as } x \rightarrow \infty$ (respectively $\p{X > x} = o(\p{L_A > x}), \text{ as } x \rightarrow \infty$). 

    Proposition~\eqref{prop:sum-hawkes-trsf} extends Lemma 5.2 in \cite{bwz19} by letting $(X, \kappa_A)$ be regularly varying, while it is shown in the aforementioned paper that $D^H$ is regularly varying in the case $X$ is itself regularly varying and with noninteger $\alpha \in (0,2)$. In the aforementioned paper, three cases are distinguished, with various assumptions on the relation between $X$ and $L_A$. Note that we do not cover the case $\alpha \in (0,1)$ in Proposition~\eqref{prop:sum-hawkes-trsf}, which is studied in \cite{bwz19}. 

    
    In a recent contribution concerning PageRank, Theorem 4.2 in \cite{oc21} provides similar asymptotics as in Theorem 4.2, that can be specialised to our case when $X$ and $K_A$ are allowed to have any form of dependence but one has a negligible tail with respect to the other. The aforementioned theorem also applies to intermediate regularly varying $X$ and $K_A$. The main connection and specialisation is the following one:
    \begin{enumerate}
        \item if $K_A$ is regularly varying with index $\alpha > 1$ and  $\E{X^{\alpha+\epsilon}} < \infty$ for some $\epsilon > 0$, and if $\p{X > x} = o(\p{K_A > x}), \text{ as } x \rightarrow \infty,$ then $$\p{D^H > x} \sim \E{K_A} \p{X > x} + \p{\E{X} K_A > x}, \text{ as } x \rightarrow \infty;$$
        \item if $X$ is regularly varying with index $\alpha > 1$ and  $\E{K_A^{\alpha+\epsilon}} < \infty$ for some $\epsilon > 0$, and if $\p{K_A > x} = o(\p{X > x}), \text{ as } x \rightarrow \infty,$ then $$\p{D^H > x} \sim (1+\E{K_A}) \p{X > x}, \text{ as } x \rightarrow \infty.$$
     \end{enumerate}
     Hence, our result essentially extends the above, allowing for tail equivalence between $X$ and $K_A$.
    
\end{rem} 

\section{Precise large deviations of cluster process functionals}

In this section, we make use of the cluster asymptotics from Section~4 to Section~7 to derive (precise) large deviation results for the renewal Poisson cluster process as well as for the Hawkes process. 

Notation wise, we let $$N_T = \big\lvert \{(i,j): 0 \leq \Gamma_i \leq T, 0 \leq \Gamma_i + T_{ij} \leq T\}\big\rvert$$ represent the number of events occurring in the time interval $[0,T]$, for $T >0,$ and we let $$J_T = \big\lvert\{(i,j): 0 \leq \Gamma_i \leq T, T \leq \Gamma_i + T_{ij}\}\big\rvert$$ represent the number of (ordered) events coming from clusters that started in the time interval $[0,T]$, but occurring after time $T>0$. We will also need the following decomposition of the maximum: for $x> 0$
\begin{equation}\label{eq:max-decomp}
   \bigg\{\max_{1 \leq i \leq C_T} H_i - \max_{1 \leq j \leq J_T} X_j > x \bigg\} \subseteq \bigg\{\max_{1 \leq i \leq N_T} X_i > x\bigg\} \subseteq \bigg\{\max_{1 \leq i \leq C_T} H_i > x\bigg\}
\end{equation}
where $C_T \sim \text{Poisson}(\nu T)$ is the number of clusters starting in the interval $[0, T]$, for $T>0$, and $H_i$ is as in Equation~\eqref{eq:hi}. This is due to the fact that the immigration process is the classical homogeneous Poisson process with parameter $\nu >0$, see Section 2. The upper bounding set in decompositions~\eqref{eq:max-decomp} overshoots by taking the maximum over all the events belonging to clusters initiated before time $T >0$, i.e. this includes events occurring after time $T>0$. This is convenient, since $C_T$ and $H$ are independent. 

The precise large deviation results for the sum will necessitate another decomposition. Notation wise, rewriting Equation~\eqref{eq:st} using $N_T$ yields: $$S_T \defeq \sum_{j=1}^{N_T} X_j,$$ and we let $\mu_{S_T}$ denote the expectation of $S_T$. Then we can decompose the deviation as: 
\begin{align}\label{eq:sum-decomp} S_T - \mu_{S_T} &= \sum_{i=1}^{C_T} D_i - \E{\sum_{i=1}^{C_T} D_i} -\bigg(\sum_{j=1}^{J_T} f(A_j)-\E{\sum_{j=1}^{J_T} f(A_j)}\bigg) \nonumber \\
 	 &\eqdef \sum_{i=1}^{C_T} D_i - \E{\sum_{i=1}^{C_T} D_i} -\big(\varepsilon_T -\E{\varepsilon_T}\big).\end{align}
As in decomposition~\eqref{eq:max-decomp}, the first difference overshoots by summing marks of all events belonging to clusters started before $T>0$, and removing the left-over effect of events occurring after time $T>0$ in a second step, denoted by $\varepsilon_T$. Again, note that $C_T$ and $D$ are independent. 

Furthermore, regarding the left-over effect, the following properties hold: 
\begin{enumerate}
    \item (Property 1) in \cite{bmz22}, for both the renewal Poisson cluster process and the Hawkes process, that $\E{J_T} = o(T), \text{ as } T \rightarrow \infty;$
    \item (Property 2) in \cite{bwz19}, for both the renewal Poisson cluster process and the Hawkes process, that $\E{\varepsilon_T} = o(\sqrt{T}), \text{ as } T \rightarrow \infty;$ and hence, in our settings, the condition  $\E{\varepsilon_T} = o(T), \text{ as } T \rightarrow \infty$ holds as well. 
\end{enumerate}

\subsection{Large deviations of maxima over an interval $[0,T]$}

We now illustrate how the asymptotics of Proposition~\eqref{prop:max-rp-trsf} and Proposition~\eqref{prop:max-hawkes-trsf} help to determine the asymptotic behaviour of the whole processes on an interval. In what follows, we let $H$ denote a generic maximum, i.e. it can either be $H^R$ or $H^D$ from Section~4 and Section~6. At the end of the section, we present some related work. 

\begin{prop}\label{prop:maxh}
    Suppose that the conditions of either Proposition~\eqref{prop:max-rp-trsf} or those of Proposition~\eqref{prop:max-hawkes-trsf} hold. Then, as $T \rightarrow \infty$, and for any $\gamma > 0$ $$\lim_{T \rightarrow \infty} \sup_{x \geq \gamma \nu T} \bigg\lvert \frac{\p{\max_{1 \leq i \leq N_T} X_i > x}}{\E{N_T} \p{X>x} }-1 \bigg\rvert = 0.$$
\end{prop}
\begin{proof}[Proof of Proposition~\eqref{prop:maxh}]
    Using decomposition~\eqref{eq:max-decomp}
\begin{align*}
	\frac{\p{\max_{1 \leq i \leq C_T} H_i - \max_{1 \leq j \leq J_T} X_i > x}}{\E{N_T} \p{X>x}} \leq \frac{\p{\max_{1 \leq i \leq N_T} X_i > x}}{\E{N_T} \p{X > x}}
	\leq \frac{\p{\max_{1 \leq i \leq C_T} H_i > x }}{\E{N_T} \p{X>x}} .
\end{align*}  
\textit{Upper bound:}
 By the remark following Theorem 3.1 in \cite{km99} for any $\gamma > 0$, $$\lim_{T \rightarrow \infty} \sup_{x \geq \gamma \nu T} \bigg\lvert \frac{\p{\max_{1 \leq i \leq C_T} H_i > x}}{\E{C_T} \p{H > x}} - 1 \bigg\rvert = 0, \text{ as } T \rightarrow \infty. $$
Using the asymptotics of Proposition~\eqref{prop:max-rp-trsf} and of Proposition~\eqref{prop:max-hawkes-trsf}, $$\E{C_T} \p{H > x} \sim \E{N_T} \p{X > x}, \text{ as } T \rightarrow \infty$$ for the $x$-values considered, i.e. when $x \geq \gamma \nu T$ for any $\gamma >0$.

\textit{Lower bound:}
\begin{align*}
    \frac{\p{\max_{1 \leq i \leq C_T} H_i - \max_{1 \leq j \leq J_T} X_j > x}}{\E{N_T} \p{X>x}}  &= \frac{\p{\max_{1 \leq i \leq C_T} H_i - \max_{1 \leq j \leq J_T} X_j > x, \max_{1 \leq j \leq J_T} X_j \leq  x\varepsilon}}{\E{N_T} \p{X>x}} \\
    &\quad + \frac{\p{\max_{1 \leq i \leq C_T} H_i - \max_{1 \leq j \leq J_T} X_j > x, \max_{1 \leq j \leq J_T} X_j >  x\varepsilon}}{\E{N_T} \p{X>x}}\\ 
        &\geq \frac{\p{\max_{1 \leq i \leq C_T} H_i   > x(1+\varepsilon), \max_{1 \leq j \leq J_T} X_j \leq  x\varepsilon}}{\E{N_T} \p{X>x}} \\
    &\geq \frac{\p{\max_{1 \leq i \leq C_T} H_i > x(1+\varepsilon)}}{\E{N_T} \p{X>x}} \\
    &\quad  - \frac{\p{\max_{1 \leq i \leq C_T} H_i  > x(1+\varepsilon),  \max_{1 \leq j \leq J_T} X_j > x\varepsilon}}{\E{N_T} \p{X>x}}.
\end{align*} 
The very last term in the lower bound is bounded above by
\begin{align*}
    \frac{\p{\max_{1 \leq i \leq C_T} H_i  > x(1+\varepsilon), \max_{1 \leq j \leq J_T} X_j > x\varepsilon}}{\E{N_T} \p{X>x}} &\leq \frac{\p{\max_{1 \leq j \leq J_T} X_j > x\varepsilon}}{\E{N_T} \p{X>x}}.
\end{align*}
Conditioning on the values of $J_T$, using a union bound and the fact that the $X_j$s are independent, 
\begin{align*}
    \p{\max_{1 \leq j \leq J_T} X_j > x \varepsilon} &= \sum_{k=1}^{\infty} \p{\max_{1 \leq j \leq k} X_j > x \varepsilon} \p{J_T = k} \\ 
    &\leq \sum_{k=1}^{\infty} \sum_{j=1}^k \p{X_j > x \varepsilon} \p{J_T=k} \\
    &\leq \sum_{k=1}^{\infty} k \p{X>x\varepsilon} \p{J_T=k} \\
    &\leq \E{J_T} \p{X>x\varepsilon}.
\end{align*}
Using Property (1) above, and Remark~\eqref{rem:nT}, which essentially says that $\E{N_T} = \cO(T), \text{ as } T \rightarrow \infty$, and under the assumption that $x \geq \gamma \nu T$ for every $\gamma>0$, it holds that $T \p{X > x} \rightarrow 0$ as $T \rightarrow \infty$, and it follows that, for any fixed $\epsilon >0$,
 $$\frac{\p{\max_{1 \leq j \leq J_T} X_j > x\varepsilon}}{\E{N_T} \p{X>x}} = o(1), \text{ as } T \rightarrow \infty.$$ This implies that 
 \begin{align*}\frac{\p{\max_{1 \leq i \leq C_T} H_i - \max_{1 \leq j \leq J_T} X_j > x}}{\E{N_T} \p{X>x}} &\geq \frac{\p{\max_{1 \leq i \leq C_T} H_i > x(1+\varepsilon)}}{\E{N_T} \p{X>x}}. \end{align*}
Using again the remark following Theorem 3.1 in \cite{km99}, it follows, for any $x \geq \gamma \nu T$, that
$$\lim_{T \rightarrow \infty} \sup_{x \geq \gamma \nu T} \bigg\lvert \frac{\p{\max_{1 \leq i \leq C_T} H_i > x(1+\varepsilon)}}{\E{C_T} \p{H > x(1+\varepsilon)}} - 1 \bigg\rvert = 0, \text{ as } T \rightarrow \infty. $$
Because $H$ is regularly varying with index $\alpha>1$, it follows that $$\E{C_T} \p{H > x(1+\varepsilon)} = (1+\varepsilon)^{-\alpha} \E{C_T}\p{H > x}, \text{ as } x \rightarrow \infty,$$ and using the asymptotics of Proposition~\eqref{prop:max-rp-trsf} and of Proposition~\eqref{prop:max-hawkes-trsf}, $$\E{C_T}\p{H > x(1+\varepsilon)} \sim (1+\varepsilon)^{-\alpha}\E{N_T}\p{X > x}, \text{ as } T \rightarrow \infty.$$ 

Letting $\epsilon \rightarrow 0$, collecting the upper and lower bounds yields the desired result. 
\end{proof}

\begin{rem}\label{rem:nT}
    Note that, by the independence of the clusters, we have: 
    \begin{enumerate}
        \item for the renewal Poisson cluster process, $\E{N_T}=(\E{K_A}+1) \nu T$;
        \item for the Hawkes process, $\E{N_T}=\frac{\nu T}{1-\E{\kappa_A}}$ (see e.g. Section 12.1 in \cite{bremaud20}). 
    \end{enumerate}
    \end{rem}

\subsection{Large deviations of sums over an interval \([0,T]\)}

We finally illustrate how the results of Proposition~\eqref{prop:sum-rp-trsf} and Proposition~\eqref{prop:sum-hawkes-trsf} help to derive results for the mixed binomial Poisson cluster process as well as for the Hawkes on an interval $[0,T]$. Note that $D$ denotes a generic sum of the marks. 

\begin{prop}\label{prop:sumh}
Suppose $\lim_{T \rightarrow \infty} \sup_{x \geq \gamma \nu T} \frac{\p{\varepsilon_T - \E{\varepsilon_T} > x}}{\nu T \p{D > x}} = 0 $ for both the mixed binomial Poisson cluster process and the Hawkes process. 
   \begin{enumerate}
       \item Suppose the conditions of Proposition~\eqref{prop:sum-rp-trsf} hold for the mixed binomial Poisson cluster process. Then, as $T \rightarrow \infty$, for all $\gamma>0$,
       $$\lim_{T \rightarrow \infty} \sup_{x \geq \gamma\nu T} \bigg\lvert \frac{\p{S_T-\mu_{S_T} > x}}{\nu T\big(\p{X+\E{X}K_A>x} + \E{K_A}\p{X>x}\big)}-1 \bigg\rvert = 0.$$
       \item Suppose the conditions of Proposition~\eqref{prop:sum-hawkes-trsf} hold. Then, as $T \rightarrow \infty$, for all $\gamma>0$, $$\lim_{T \rightarrow \infty} \sup_{x \geq \gamma\nu T} \bigg\lvert \frac{\p{S_T-\mu_{S_T} > x}}{\frac{\nu T}{1-\E{\kappa_A}} \p{X+\bigg(\frac{\E{X}}{1-\E{\kappa_A}}\bigg)\kappa_A>x} }-1 \bigg\rvert = 0.$$
   \end{enumerate}
\end{prop}

\begin{proof}[Proof of Proposition~\eqref{prop:sumh}]

We use decomposition~\eqref{eq:sum-decomp}, i.e. 
$$\p{S_T-\mu_{S_T} > x} = \p{\sum_{i=1}^{C_T} D_i - \E{\sum_{i=1}^{C_T} D_i} -\big(\varepsilon_T -\E{\varepsilon_T}\big) > x}.$$

\textit{Upper bound:}
Note that
\begin{align*}
	\p{S_T-\mu_{S_T} > x} &\leq \p{\sum_{i=1}^{C_T} D_i - \E{\sum_{i=1}^{C_T} D_i}  > x - \E{\varepsilon_T}}. 
\end{align*}

As $T \rightarrow \infty$, we can rewrite $x \geq \gamma \nu T$ as $x \geq \gamma^{'} \nu T + \E{\varepsilon_T}$, for some $0 < \gamma^{'} < \gamma$. Hence, under the assumption that $x \geq \gamma^{'} \nu T + \E{\varepsilon_T}$, then $x-\E{\varepsilon_T} \geq  \gamma^{'} \nu T$, and since $C_T \sim \text{Poisson}(\nu T)$ is independent of $D$, using Lemma 2.1 and Theorem 3.1 in \cite{km99} yields
$$\lim_{T \rightarrow \infty} \sup_{x \geq \gamma^{'} \nu T} \bigg\lvert \frac{\p{\sum_{i=1}^{C_T} D_i - \E{\sum_{i=1}^{C_T} D_i}  > x - \E{\varepsilon_T}}}{\nu T \p{D > x - \E{\varepsilon_T}}} -1 \bigg\rvert = 0.$$

Recall that $D$ is regularly varying with index $\alpha>1$. Using Property (2) above, we can write $x-\E{\varepsilon_T} = x - o(T)$ as $T \rightarrow \infty$. Using the Potter bounds (see Theorem 1.5.6 in \cite{bgt89}), for all $I > 1$, $\eta > 0$, there exists $X$ such that, for all $x-o(T) \geq X$, 
\begin{align*}
	\frac{\p{D > x-o(T)}}{\p{D>x}} &\leq I \max \bigg\{\bigg(1-\frac{o(T)}{x}\bigg)^{-\alpha + \eta},\bigg(1-\frac{o(T)}{x}\bigg)^{-\alpha + \eta}\bigg\}.
\end{align*}

Because $x \geq \gamma \nu T + \E{\varepsilon_T}$, the above upper bound becomes uniformly close to 1, as $T \rightarrow \infty$. In combination with the above, it follows that, as $T \rightarrow \infty$, uniformly for $x \geq \gamma^{'} \nu T + \E{\varepsilon_T}$,
$$\p{S_T-\mu_{S_T} > x} \leq \nu T \p{D > x}.$$ 

\textit{Lower bound:}
Let $\delta >0$, and note that 
\begin{align*}
	\p{S_T-\mu_{S_T} > x} &= \p{\sum_{i=1}^{C_T} D_i - \E{\sum_{i=1}^{C_T} D_i} -\big(\varepsilon_T -\E{\varepsilon_T}\big) > x, \hspace{0.15cm} \varepsilon_T - \E{\varepsilon_T} \leq x \delta} \\
	&\quad +\p{\sum_{i=1}^{C_T} D_i - \E{\sum_{i=1}^{C_T} D_i} -\big(\varepsilon_T -\E{\varepsilon_T}\big) > x, \hspace{0.15cm} \varepsilon_T - \E{\varepsilon_T} > x \delta} \\
	&\geq \p{\sum_{i=1}^{C_T} D_i - \E{\sum_{i=1}^{C_T} D_i} -\big(\varepsilon_T -\E{\varepsilon_T}\big) > x, \hspace{0.15cm} \varepsilon_T - \E{\varepsilon_T} \leq x \delta} \\
	&\geq \p{\sum_{i=1}^{C_T} D_i - \E{\sum_{i=1}^{C_T} D_i} > x(1+\delta)} \\
	&\quad - \p{\sum_{i=1}^{C_T} D_i - \E{\sum_{i=1}^{C_T} D_i} -\big(\varepsilon_T -\E{\varepsilon_T}\big) > x, \hspace{0.15cm} \varepsilon_T - \E{\varepsilon_T} > x \delta} \\
	&\geq \p{\sum_{i=1}^{C_T} D_i - \E{\sum_{i=1}^{C_T} D_i} > x(1+\delta)} - \p{\varepsilon_T - \E{\varepsilon_T} > x \delta}.
\end{align*}    

By assumption, the second term is (uniformly) negligible with respect to $\nu T \p{D > x}$ for the $x$-region considered. 

Since $x \geq \gamma^{'} \nu T + \E{\varepsilon_T} \geq \gamma \nu T$, using again Theorem 3.1 in \cite{km99}, it follows that
$$\lim_{T \rightarrow \infty} \sup_{x \geq \gamma \nu T} \bigg\lvert \frac{\p{\sum_{i=1}^{C_T} D_i - \E{\sum_{i=1}^{C_T} D_i}  > x(1+\delta)}}{\nu T \p{D > x(1+\delta)}} -1 \bigg\rvert = 0.$$

Since $D$ is regularly varying with index $\alpha>1$, letting $\delta \rightarrow 0$ yields $$\nu T \p{D > x(1+\delta)} \sim \nu T(1+\delta)^{-\alpha} \p{D > x} \sim \nu T \p{D > x}.$$ It follows that, uniformly for $x \geq \gamma \nu T$, and as $T \rightarrow \infty$, $$ \nu T \p{D> x} \leq \p{S_T - \mu_{S_T} > x}.$$ 

Collecting the above upper and lower bounds, 

\begin{enumerate}
    \item the asymptotics of Proposition~\eqref{prop:sum-rp-trsf} yields
    $$\lim_{T \rightarrow \infty} \sup_{x \geq \gamma\nu T} \bigg\lvert \frac{\p{S_T-\mu_{S_T} > x}}{\nu T\big(\p{X+\E{X}K_A>x} + \E{K_A}\p{X>x}\big)}-1 \bigg\rvert = 0.$$
    \item the asymptotics of Proposition~\eqref{prop:sum-hawkes-trsf} yields 
    $$\lim_{T \rightarrow \infty} \sup_{x \geq \gamma\nu T} \bigg\lvert \frac{\p{S_T-\mu_{S_T} > x}}{\E{N_T} \p{X+\bigg(\frac{\E{X}}{1-\E{\kappa_A}}\bigg)\kappa_A>x} }-1 \bigg\rvert = 0,$$ and recalling that $\frac{\nu T}{1-\E{\kappa_A}} = \E{N_T}$, this concludes the proof.
\end{enumerate}
\end{proof}

\begin{rem}
    Early contributions to the (non-uniform) precise large deviations results for non-random sums of i.i.d. regularly varying random variables can be found in \cite{n69}, \cite{n69-2}, \cite{heyde67}, or \cite{n79}. 
    
    The proofs of Proposition~\eqref{prop:maxh} and Proposition~\eqref{prop:sumh} heavily rely on the work of \cite{km99}, in which the authors show that, under the assumption that the process of integer-valued non-negative random variables $(N_T)_{T > 0}$ is such that
    \begin{enumerate}
        \item $N_T/\lambda_T \cvgpr 1, \text{ as } \lambda_T \rightarrow \infty$, where $\lambda_T = \E{N_T}$;
        \item the following limit holds: $$\sum_{k > (1+\delta) \lambda_T} \p{N_T > k}(1+\epsilon)^k \rightarrow 0, \text{ as } \lambda_T \rightarrow \infty.$$
    \end{enumerate}
    Furthermore, if the process $(N_T)$ is independent of the sequence $(X_j)$, by their Theorem 3.1, if the distribution of $X$ is extended regularly varying, for any $\gamma > 0$, $$\lim_{T \rightarrow \infty} \sup_{x \geq \gamma \lambda_T} \bigg\lvert \frac{\p{S_T-\mu_{S_T} > x}}{\lambda_T \p{X > x}}-1 \bigg\rvert = 0, \text{ and } \lim_{T \rightarrow \infty} \sup_{x \geq \gamma \lambda_T} \bigg\lvert \frac{\p{\max_{1 \leq j \leq N_T} X_j > x}}{\lambda_T \p{X > x}}-1 \bigg\rvert = 0$$ where $S_T = \sum_{j=1}^{N_T} X_j.$ Note that the authors show that the Poisson process $C_T$ satisfies the assumptions above, but the second condition is difficult to show for more complicated processes. Hence, the trick is to bound the processes at hand in this work by a process governed by an independent variable, in our context $C_T$ which is Poisson distributed and satisfies the settings of \cite{km99}. 

    Note that the work in \cite{km99} extends the precise large deviation principles already studied in \cite{ch91} (in the case of non-random sums) to the case of random sums.
    
    In \cite{tsjz01}, the authors relax the two assumptions used in \cite{km99} and mentioned above, and reduce them into the single condition that $$\E{N_T^{\beta + \epsilon} \I{N_T > (1+\delta)\lambda_T}} = \cO(\lambda_T), \text{ as } T \rightarrow \infty,$$ for fixed $\epsilon, \delta > 0$ small and $\beta$ the (upper) index of extended regular variation, and prove similar precise large deviation results as \cite{km99}. In \cite{nty04}, the authors study another subclass of the subexponential family, namely the consistently varying random variables, and prove similar precise large deviations under the same conditions as \cite{tsjz01}. 

    Under the assumption that the sequence $(X_j)$ exhibits negative dependence, i.e. $$\p{\bigcap_{j=1}^n \{X_j \leq x_j \}} \leq M \prod_{j=1}^n \p{X_j \leq x_j} \text{ and } \p{\bigcap_{j=1}^n \{X_j > x_j \}} \leq M \prod_{j=1}^n \p{X_j > x_j}$$ for some $M > 0$, all $x_1, \ldots, x_n \in \R$, more recent literature such as \cite{tang06} or \cite{liu09} propose extensions and similar results to those of \cite{nty04} under the same consistently varying random variables.   

    While our framework is more restrictive on the aspect that our sequence $(X_j)_{1 \leq j \leq N_T}$ has elements that are regularly varying, which is a subclass of the extended regularly varying distributions, and that furthermore the elements of the sequence are independent, knowledge of the tail asymptotics of the cluster functionals allowed us to derive expressions that resemble known precise large deviations principles for random maxima and sums of independent random variables, even though, clearly, $N_T$ and $(X_j)$ are dependent over a time window $[0,T]$. This comes at the cost of an extra term, for the sums the marks over a finite time interval, of an extra left-over effect $\E{\varepsilon_T}$ that vanishes as $T$ becomes large. 
    
\end{rem}

\appendix 
\section{}\label{app:1}
\subsection*{Proof of Proposition~\eqref{prop:max-rp-trsf}}

\begin{proof}[Proof of Proposition~\eqref{prop:max-rp-trsf}]
By conditioning and using the independence of $X$ and $X_j$, $j\ge 1$, and that of $K_A$ and $X_j$, $j\ge 1$, we obtain
\begin{align}
\p{H^{R}>x} &= 1-\sum_{k=0}^{\infty} \mathbb{P} \big(X \leq x \; \vert \; K_A=k \big) \big(\p{X \leq x}\big)^k\p{K=k} \nonumber \\
&= 1-\sum_{k=0}^{\infty} \mathbb{P} \big(X \leq x \; \vert \; K_A=k \big) \exp\big(k \log \big(1-\p{X>x} \big) \big)\p{K=k}. \label{eq:h}
\end{align}

    A Taylor expansion on the exponential term, as $x \rightarrow \infty$ (and hence, as $\p{X > x} \rightarrow 0$ by the integrability of $X$), gives 
    \begin{align*}
    \exp\big(k \log \big(1-\p{X>x} \big) \big) &= \exp\big(-k\p{X>x} - o\big(k\p{X>x}\big)\big) \\ 
&= \big(1-k\p{X>x} + o\big(k\p{X>x}\big)\big) \exp \big(- o\big(k\p{X>x}\big) \big) 
\end{align*}
where the last equality follows by another Taylor expansion of the first exponential term in the second equality, as $x \rightarrow \infty$. 

Plugging the above expansion in Equation~\eqref{eq:h} yields
\begin{align*}
\p{H^{R}>x} &= 1-\sum_{k=0}^{\infty} \p{X \leq x, K_A=k} \exp \big(- o\big(k\p{X>x}\big) \big) \\ 
&\quad + \p{X>x}\sum_{k=0}^{\infty} k \p{X \leq x, K_A=k} \exp \big(- o\big(k\p{X>x}\big) \big)\\ 
&\quad - o(\p{X>x})\sum_{k=0}^{\infty} \p{X \leq x, K_A=k} \exp \big(- o\big(k\p{X>x}\big) \big) \\
&\eqdef 1- B_1 + B_2 - B_3.
\end{align*}

We treat each term separately. For term $1-B_1$, remarking that $1 = \p{X \leq x} + \p{X > x}$, we obtain
\begin{align*} 1-B_1 &= \p{X > x} + \sum_{k=0}^{\infty} \p{X \leq x, K_A = k} \big(1-\exp \big(- o(k\p{X>x})\big)\big)
\end{align*}
Using the basic inequality $1-e^{-x} \leq x$, term $1-B_1$ is bounded by $$0 \leq 1-B_1 \leq \p{X > x} + o(\E{K_A}\p{X>x}),  \text{ as } x \rightarrow \infty.$$ 

For term $B_2$, we can write 
\begin{align*}
    B_2 &= \p{X > x} \sum_{k=0}^{\infty} k \big( \p{X \leq x, K_A=k} \exp\big(- o(k\p{X > x})\big) + \p{K=k}  - \p{K=k} \big) \\
    &= \p{X > x} \E{K_A} - \p{X > x} \sum_{k=0}^{\infty} k \p{X > x, K_A = k} \\ & \quad + \p{X >x} \sum_{k=0}^{\infty} k \p{X \leq x, K_A = k}  \big(\exp\big(- o(k\p{X > x})\big) - 1 \big) \\
    &\eqdef B_{21} - B_{22} + B_{23}.
 \end{align*}
Note that $B_{22}$ is bounded above by $B_{22} \leq \E{K_A} \p{X > x},$ and hence, by a dominated convergence argument and the integrability of $X$, we have that $$ \p{X > x} \sum_{k=0}^{\infty} k \p{X > x, K_A=k} = o\big(\p{X>x}\big), \text{ as } x \rightarrow \infty.$$ 

For term $B_{23}$, which is negative since for all $k \geq 0$, $0 \leq e^{-o(k\p{X>x})} \leq 1$, we bound it below by 
\begin{align*}- \p{X > x} \sum_{k=0}^{\infty} k & \p{X \leq x, K_A = k} \\ &\leq \p{X >x} \sum_{k=0}^{\infty} k \p{X \leq x, K_A = k}  \big(\exp\big(- o(k\p{X > x})\big) - 1 \big) \\ &\qquad \qquad \qquad \leq 0,\end{align*} and hence, by a dominated convergence argument, we obtain that, as $x \rightarrow \infty$, $$\p{X > x} \sum_{k=0}^{\infty} k \p{X \leq x, K_A = k}  \big(\exp\big(- o(k\p{X > x})\big) - 1 \big) = o\big(\p{X>x}\big).$$ Collecting the above results, we see that, essentially, $$B_2 = \p{X>x}\E{K_A} + o\big(\p{X>x}\big),  \text{ as } x \rightarrow \infty.$$
Finally, by very similar arguments to those employed for $B_2$ and omitted for brevity, $$B_3 = o\big(\p{X>x}\big),  \text{ as } x \rightarrow \infty.$$

Collecting the above, it essentially follows that $$P(H^{R} >x)=\p{X>x}+\p{X > x}\E{K_A} + o\big(\p{X>x}\big), \text{ as } x \rightarrow \infty.$$ The desired result follows at once by taking the limit, as $x \rightarrow \infty$, and upon using the assumption that the limiting Radon measure is non-null on the subspace $\{(x_1, x_2) \in \R^2_{+, \bZ}: x_1 > 1\}$, which implies by means of Example~\eqref{ex:proj} that $H^{R}$ is regularly varying with index $\alpha > 1$. 
\end{proof}

\section{}\label{app:2}
\subsection*{Proof of Proposition~\eqref{prop:max-hawkes-trsf}}
\begin{proof}[Proof of Proposition~\eqref{prop:max-hawkes-trsf}]
By conditioning and using the independence of $X$ and $H^{H}$, and that of $L_A$ and $H^{H}$, we obtain as in the proof of Proposition~\eqref{prop:max-rp-trsf}
\begin{align*}
\p{H^{H}>x} &= 1-\sum_{k=0}^{\infty} \mathbb{P} \big(X \leq x \; \vert \; L_A=k \big) \exp\big(k \log \big(1-\p{H^{H}>x} \big) \big)\p{L_A=k}.
\end{align*}
     A Taylor expansion on the exponential term, as $x \rightarrow \infty$ (and hence, as $\p{H^{H} > x} \rightarrow 0$ by the integrability of $H^{H}$), yields, as $x \rightarrow \infty$, 
     \begin{multline*}
         \exp\big(k \log \big(1-\p{H^{H}>x}\big) \big) \\ = \big(1-k\p{H^{H}>x} + o\big(k\p{H^{H}>x}\big)\big) \exp \big(- o\big(k\p{H^{H}>x}\big) \big).
     \end{multline*}

From here on, the proof follows the same lines as that of Proposition~\eqref{prop:sum-rp-trsf}, except that the tail of $H^{H}$ appears here rather than the tail of $X$. The proof is omitted for brevity, but we retrieve 
$$P(H^{H} >x)=\p{X>x}+\E{L_A}\p{H^{H} > x} + o\big(\p{H^{H}>x}\big), \text{ as } x \rightarrow \infty$$
which yields the desired result. 
\end{proof}

\section{}\label{app:3}
\subsection*{Proof of Proposition~\eqref{prop:sum-rp-trsf}}
We need the following Lemma in order to prove Lemma~\eqref{lem:rvcomp} used in the proof of Proposition~\eqref{prop:sum-rp-trsf}: 
\begin{lem}\label{lem:negli}
    Suppose $(X, K_A)$ is regularly varying with index $\alpha \in (n, n+1)$, for $n \in \N$. Additionally, suppose that $X$ has negligible tails with respect to $X + \E{X}K_A$, i.e. $\p{X > x} = o(\p{X + \E{X}K_A > x})$, as $x \rightarrow \infty.$ Then, $$\varphi_{X}^{(n+1)}(s) = o\big(\varphi_{X+\E{X}K_A}^{(n+1)}(s)\big), \text{ as } s \rightarrow 0^+.$$
\end{lem}
\begin{proof}
    Note that \begin{align*}\varphi_{X}^{(n+1)}(s) &= \E{(-X)^{n+1} e^{-sX}} = \int_0^{\infty} x^{n+1} e^{-sx} \diff(-\p{X \geq x}) \\
    &= \big[-x^{n+1}e^{-sx} \p{X \geq x}\big]_{0}^{\infty} + \int_{0}^{\infty} \big((n+1) x^{n} e^{-sx} - sx^{n+1} e^{-sx} \big) \p{X \geq x} \diff x.
    \end{align*}
    The first term above vanishes; upon substituting, the second term yields 
    \begin{multline*}
         \int_{0}^{\infty} \big((n+1) x^{n} e^{-sx} - sx^{n+1} e^{-sx} \big) \p{X \geq x} \diff x \\ = \int_{0}^{\infty} \big((n+1) (y/s)^{n} e^{-y} - s(y/s)^{n+1} e^{-y} \big) \p{X \geq y/s} \frac{\diff y}{s} \\
        = s^{-(n+1)} \int_{0}^{\infty} \big((n+1) y^{n} e^{-y} - y^{n+1} e^{-y} \big) \p{X \geq y/s} \diff y.
    \end{multline*}
    Fix $\varepsilon > 0$ small and split the above integral into 
    \begin{multline*}
        s^{-(n+1)} \int_{0}^{\infty} \big((n+1) y^{n} e^{-y} - y^{n+1} e^{-y} \big) \p{X \geq y/s} \diff y \\ 
        = s^{-(n+1)}\bigg(\int_0^{\varepsilon} \big((n+1) y^{n} e^{-y} - y^{n+1} e^{-y} \big) \p{X \geq y/s} \diff y \\
         + \int_{\varepsilon}^{\infty} \big((n+1) y^{n} e^{-y} - y^{n+1} e^{-y} \big) \p{X \geq y/s} \diff y \bigg) 
    \eqdef I_1 + I_2. 
    \end{multline*}

Consider integral $I_2$ first. For some values $y \in [\varepsilon, \infty)$ the expression $(n+1) y^{n} e^{-y} - y^{n+1} e^{-y}$ might be negative, so bound $I_2$ above by its absolute value. Additionally, upon using the hypothesis of negligibility of the tail of $X$ with respect to the tail of $X+\E{X} K_A$, it follows that, for any $\delta >0$, for any fixed $\varepsilon>0$ and $y > \varepsilon$, there is $s_0$ such that for all $s \leq s_0$, $\p{X \geq y/s} \leq \delta_{\varepsilon} \p{X+ \E{X}K_A > y/s}.$ All in all, because $X + \E{X}K_A$ is regularly varying with index $\alpha \in (n, n+1)$, this yields as an upper bound 
\begin{align*}
    \lvert I_2 \rvert &\leq s^{-(n+1)} \int_{\varepsilon}^{\infty} \big\lvert  \big((n+1) y^{n} e^{-y} - y^{n+1} e^{-y} \big) \big\rvert \delta_{\varepsilon} \p{X + \E{X}K_A > y/s} \diff y \\
    &\leq s^{-(n+1)} \int_{\varepsilon}^{\infty} \big\lvert \big((n+1) y^{n} e^{-y} - y^{n+1} e^{-y} \big)\big\rvert  \delta_{\varepsilon} (y/s)^{-\alpha} L_{X + \E{X}K_A}(y/s) \diff y \\
    &\leq s^{\alpha-(n+1)} \delta_{\varepsilon} \int_{\varepsilon}^{\infty}\big\lvert  \big((n+1) y^{n} e^{-y} - y^{n+1} e^{-y} \big)\big\rvert  y^{-\alpha} L_{X + \E{X}K_A}(y/s) \diff y. 
\end{align*}
Because $((n+1) y^{n} e^{-y} - y^{n+1} e^{-y} \big) y^{-\alpha}$ is integrable over $[0, \infty)$, it follows from Proposition 4.1.2 (b) in \cite{bgt89} that, as $s \rightarrow 0^+$,
\begin{multline*}
    s^{\alpha-(n+1)} \delta_{\varepsilon} \int_{\varepsilon}^{\infty} \big\lvert \big((n+1) y^{n} e^{-y} - y^{n+1} e^{-y} \big)\big\rvert  y^{-\alpha} L_{X + \E{X}K_A}(y/s) \diff y \\ \sim s^{\alpha-(n+1)} L_{X + \E{X}K_A}(1/s) \delta_{\varepsilon} \int_{\varepsilon}^{\infty} \big\lvert  \big((n+1) y^{n} e^{-y} - y^{n+1} e^{-y} \big)\big\rvert  y^{-\alpha} \diff y.
\end{multline*}

For each fixed value of $\varepsilon>0$, and as $s \rightarrow 0^+$, it is possible to take $\delta_{\varepsilon}>0$ as small as needed so that to guarantee that 
$$\delta_{\varepsilon}\int_{\varepsilon}^{\infty} \big\lvert \big((n+1) y^{n} e^{-y} - y^{n+1} e^{-y} \big)\big\rvert  y^{-\alpha} \diff y = o(1) \text{ as } s \rightarrow 0^+.$$

This implies that, for a fixed $\epsilon>0$, as $s \rightarrow 0^+$, 
\begin{align*}
    \frac{\lvert I_2 \rvert}{s^{\alpha-(n+1)} L_{X + \E{X}K_A}(1/s)} &\leq \frac{ s^{\alpha-(n+1)} L_{X + \E{X}K_A}(1/s) \delta_{\varepsilon} \int_{\varepsilon}^{\infty} \big\lvert \big((n+1) y^{n} e^{-y} - y^{n+1} e^{-y} \big)\big\rvert y^{-\alpha} \diff y}{s^{\alpha-(n+1)} L_{X + \E{X}K_A}(1/s)} \\
&\leq o(1).
\end{align*}

Consider now integral $I_1$. Because $X$ is stochastically dominated by $X + \E{X}K_A$, and using the regular variation of the latter quantity, this yields 
\begin{align*}
    \lvert I_1 \rvert &\leq s^{-(n+1)} \int_{0}^{\varepsilon} \big\lvert \big((n+1) y^{n} e^{-y} - y^{n+1} e^{-y} \big)\big\rvert \p{X + \E{X}K_A > y/s} \diff y \\
    &\leq s^{\alpha-(n+1)} \int_{0}^{\varepsilon} \big\lvert\big((n+1) y^{n} e^{-y} - y^{n+1} e^{-y} \big)\big\rvert y^{-\alpha} L_{X + \E{X}K_A}(y/s) \diff y.
\end{align*}
Because the function $\big((n+1) y^{n} e^{-y} - y^{n+1} e^{-y} \big) y^{-\alpha}$ is integrable over $[0, \varepsilon)$, it follows by Proposition 4.1.2 (a) in \cite{bgt89} that, as $s \rightarrow 0^+$, 
\begin{multline*}
    s^{\alpha-(n+1)} \int_{0}^{\varepsilon} \big\lvert \big((n+1) y^{n} e^{-y} - y^{n+1} e^{-y} \big)\big\rvert y^{-\alpha} L_{X + \E{X}K_A}(y/s) \diff y \\ \sim s^{\alpha-(n+1)} L_{X + \E{X}K_A}(1/s) \int_{0}^{\varepsilon} \big\lvert\big((n+1) y^{n} e^{-y} - y^{n+1} e^{-y} \big)\big\rvert y^{-\alpha} \diff y.
\end{multline*}
It follows that 
\begin{align*}
    \frac{\lvert I_1 \rvert}{s^{\alpha-(n+1)} L_{X + \E{X}K_A}(1/s)} &\leq \frac{s^{\alpha-(n+1)} L_{X + \E{X}K_A}(1/s) \int_{0}^{\varepsilon} \big\lvert\big((n+1) y^{n} e^{-y} - y^{n+1} e^{-y} \big)\big\rvert y^{-\alpha} \diff y}{s^{\alpha-(n+1)} L_{X + \E{X}K_A}(1/s)} \\
&\leq \int_{0}^{\varepsilon} \big\lvert\big((n+1) y^{n} e^{-y} - y^{n+1} e^{-y} \big)\big\rvert y^{-\alpha} \diff y
\end{align*}
and, because one can take $\varepsilon >0$ as small as needed, this shows that 
    $$I_1 + I_2 = o\big(s^{\alpha-(n+1)} L_{X + \E{X}K_A}(1/s)\big), \text{ as } s \rightarrow 0^+.$$
Finally, because $n+1 = \lceil \alpha \rceil$, and using Karamata's Tauberian Theorem~\eqref{thm:kar}, which implies that $\varphi_{X+\E{X}K_A}^{(n+1)} \sim C_{\alpha} s^{\alpha - \lceil \alpha \rceil} L_{X+\E{X}K_A}(1/s), \text{ as } s \rightarrow 0^+$, this shows that $$\varphi_{X}^{(n+1)}(s) = o(\varphi_{X+\E{X}K_A}^{(n+1)}), \text{ as } s \rightarrow 0^+.$$ \end{proof}

\begin{lem}\label{lem:rvcomp} 
    Suppose $(X, K_A)$ is regularly varying with index $\alpha \in (1,2)$ and slowly varying function $L_{X + \E{X}K_A}(\cdot)$ and $X$ has a negligible tail compared to the modulus $X + \E{X}K_A$, i.e. $\p{X > x} = o(\p{X + \E{X}K_A > x}), \text{ as } x \rightarrow \infty$. Then, 
   $$\frac{\varphi_X^{(1)}(s)+\E{X}}{s} = o\big( \varphi_{X + \E{X}K_A}^{(2)}(s) + \E{K_A}\varphi_X^{(2)}(s)\big), \text{ as } s \rightarrow 0^+.$$
\end{lem}
\begin{proof}
    Let $\alpha \in (1,2)$. We assess 
    \begin{align*}
    \frac{\varphi_X^{(1)}(s)+\E{X}}{s} &= \E{\frac{X\big(1-e^{-sX}\big)}{s}} \\
    &= \int_{0}^{\infty} \frac{x(1-e^{-sx})}{s} \diff (-\p{X \geq x)} \\
    &= \bigg[-\frac{x(1-e^{-sx})}{s} \p{X \geq x} \bigg]^{\infty}_{0} + \int_{0}^{\infty} \bigg(\frac{\big(1-e^{-sx}\big)}{s} + x e^{-sx} \bigg) \p{X \geq x} \diff x. 
    \end{align*}
 Since $X$ is integrable, one has that $x \p{X \geq x} = o(1) $, as $x \rightarrow \infty$, so that the first expression on the right-hand side above vanishes; for the second integral, fix $\varepsilon > 0$ small and write
	\begin{align*} 
		\int_{0}^{\infty} \bigg(\frac{\big(1-e^{-sx}\big)}{s} + x e^{-sx} \bigg) \p{X \geq x} \diff x &= \int_0^{\infty} s^{-2}(1-e^{-y}+ye^{-y}) \p{X \geq y/s} \diff y, \\
   &= \int_0^{\varepsilon} s^{-2}(1-e^{-y}+ye^{-y}) \p{X \geq y/s} \diff y \\
  &\quad + \int_{\varepsilon}^{\infty} s^{-2}(1-e^{-y}+ye^{-y}) \p{X \geq y/s} \diff y  \\
  &\eqdef (I_1 + I_2). 
  \end{align*}

Consider integral $I_2$ first. A similar argument as in the proof of Lemma~\eqref{lem:negli} for integral $I_2$ there yields the following upper bound
  \begin{align*}
     I_2 &\leq \int_{\varepsilon}^{\infty} s^{-2}(1-e^{-y}+ye^{-y}) \delta_{\varepsilon} \p{X + \E{X}K_A > y/s} \diff y \\
     &\leq s^{\alpha-2} \delta_{\varepsilon} \int_{\varepsilon}^{\infty} (1-e^{-y}+ye^{-y})y^{-\alpha} L_{X + \E{X}K_A}(y/s) \diff y.
 \end{align*}
As $\varepsilon \rightarrow 0$, the above integral diverges. But for a (small) fixed value of $\varepsilon>0$, upon using Proposition 4.1.2 (b) in \cite{bgt89}, as $s \rightarrow 0^+$, 
\begin{multline*}
    s^{\alpha-2}\int_{\varepsilon}^{\infty} (1-e^{-y}+ye^{-y})y^{-\alpha} \delta_{\varepsilon} L_{X + \E{X}K_A}(y/s) \diff y \\ 
    \sim s^{\alpha-2} L_{X + \E{X}K_A}(1/s)  \delta_{\varepsilon} \int_{\varepsilon}^{\infty} (1-e^{-y}+ye^{-y})y^{-\alpha} \diff y.
\end{multline*}

As $s \rightarrow 0^+$, and as in the proof of Lemma~\eqref{lem:negli}, it is possible to take $\delta_{\varepsilon} > 0$ as small as needed in order to ensure that
$$\delta_{\varepsilon}\int_{\varepsilon}^{\infty} (1-e^{-y}+ye^{-y}) y^{-\alpha} \diff y = o(1) \text{ as } s \rightarrow 0^+.$$

This implies that, as $s \rightarrow 0^+$,
\begin{align*}
    \frac{I_2}{s^{\alpha-2} L_{X + \E{X}K_A}(1/s)} &\leq\frac{s^{\alpha-2} L_{X + \E{X}K_A}(1/s) \delta_{\varepsilon} \int_{\varepsilon}^{\infty} (1-e^{-y}+ye^{-y})y^{-\alpha}  \diff y}{s^{\alpha-2} L_{X + \E{X}K_A}(1/s)} \leq o(1).
\end{align*}
Consider now integral $I_1$. Because $X$ is stochastically dominated by $X + \E{X}K_A$, for any fixed $\varepsilon>0$, we have 
\begin{align*}
    I_1 &\leq \int_{0}^{\varepsilon} s^{-2} (1-e^{-y}+ye^{-y}) \p{X + \E{X} K_A \geq y/s} \diff y \\
    &= \int_0^{\varepsilon} s^{\alpha-2}(1-e^{-y}+ye^{-y}) y^{-\alpha} L_{X + \E{X}K_A}(y/s) \diff y.
\end{align*}

  A Taylor expansion on the function $f(y)=e^{-y}+ye^{-y}$ yields $1 = e^{-y}-ye^{-y}+2ye^{-y}-y^2e^{-y} + o(-y)$, and we get that 
$$   \int_0^{\varepsilon} s^{\alpha-2}(1-e^{-y}+ye^{-y}) y^{-\alpha}  L_{X + \E{X}K_A}(y/s) \diff y  \approx  s^{\alpha-2}\int_0^{\varepsilon} (2ye^{-y}-y^2e^{-y}) y^{-\alpha} L_{X + \E{X}K_A}(y/s) \diff y.
$$
Because the integral $\int_{0}^{\varepsilon} (2ye^{-y}-y^2e^{-y})y^{-\alpha} \diff y < \infty$ for $\alpha \in (1,2)$ and $\varepsilon > 0$ small, even if it is potentially large for values of $\alpha$ close to 2, it follows from Proposition 4.1.2. (a) in \cite{bgt89} that, as $s \rightarrow 0^+$,
$$
     s^{\alpha-2} \int_0^{\varepsilon} (2ye^{-y}-y^2e^{-y}) y^{-\alpha}L_{X + \E{X}K_A}(y/s) \diff y   \sim s^{\alpha-2} L_{X + \E{X}K_A}(1/s) \int_0^{\varepsilon} (2ye^{-y}-y^2e^{-y}) y^{-\alpha} \diff y.
$$
Hence, as $s \rightarrow 0^+$
\begin{align*}
    \frac{I_1}{s^{\alpha-2}L_{X + \E{X}K_A}(1/s)} &\leq \frac{s^{\alpha-2} L_{X + \E{X}K_A}(1/s) \int_0^{\varepsilon} (2ye^{-y}-y^2e^{-y}) y^{-\alpha} \diff y}{s^{\alpha-2}L_{X + \E{X}K_A}(1/s)} \\
    &\leq \int_0^{\varepsilon} (2ye^{-y}-y^2e^{-y}) y^{-\alpha} \diff y
\end{align*}
and because one can take $\varepsilon > 0$ as small as needed, it essentially follows, all in all, that $$(I_1 + I_2) = o(s^{\alpha-2} L_{X + \E{X}K_A}(1/s)), \text{ as } s \rightarrow 0^+.$$ 

Because $X + \E{X}K_A$ is regularly varying, by Karamata's Tauberian Theorem~\eqref{thm:kar}, 
  \begin{align*}
        \frac{s^{\alpha-2} L_{X + \E{X}K_A}(1/s)}{\varphi_{X+\E{X}K_A}^{(2)}(s)+\E{K_A}\varphi_X^{(2)}(s)} 
		&\sim \frac{s^{\alpha-2} L_{X + \E{X}K_A}(1/s)}{C_{\alpha}s^{\alpha-2}L_{X+\E{X}K_A}(1/s) + \E{K_A} \varphi_X^{(2)}(s)}, \text{ as } s \rightarrow 0^+.
    \end{align*}
Applying the result of Lemma~\eqref{lem:negli}, this yields that $$s^{\alpha-2} L_{X + \E{X}K_A}(1/s)= \cO \big(\varphi_{X+\E{X}K_A}^{(2)}(s)+\E{K_A}\varphi_X^{(2)}(s) \big), \text{ as } s \rightarrow 0^+.$$ This yields the desired result. 
\end{proof}





\acks 
\noindent The authors would like to thank the two anonymous referees for their suggestions that helped to shorten the present paper, for pointing out unexplored relevant references, and for numerous helpful comments making it more readable. The authors would also like to acknowledge the French Agence Nationale de la Recherche (ANR) and the project with reference ANR-20-CE40-0025-01 (T-REX project), and more specifically the members of the T-REX project for organising the VALPRED3 and VALPRED4 workshops at the CNRS Centre Paul Langevin in Aussois, during which fruitful discussions led to great improvement of this article. Finally, the authors would like to thank the two anonymous referees for their careful reading of our work and their comments which greatly improved the readability of the present article. 
\fund 
\noindent The are no funding bodies to thank relating to the creation of this article.

\competing 
\noindent There were no competing interests to declare which arose during the preparation or publication process of this article.

%
%
%
%

\bibliography{mybib}
\bibliographystyle{APT}
\footnotesize

\end{document}